\documentclass[a4paper,11pt]{amsart}

\usepackage{amssymb,amsmath,amsthm,mathrsfs,enumerate,graphicx, color}
\usepackage[pdfpagelabels,colorlinks,linkcolor=blue,citecolor=black,urlcolor=blue]{hyperref}
\usepackage{esint}
\usepackage{tikz}
\usetikzlibrary{arrows}
\usepackage{citeref}

\newtheorem{thm}{Theorem}[section]

\newtheorem{cor}[thm]{Corollary}
\newtheorem{lem}[thm]{Lemma}
\newtheorem{prop}[thm]{Proposition}
\newtheorem{defn}[thm]{Definition}
\newtheorem{rem}[thm]{Remark}

\newcommand{\ip}[1]{\langle #1 \rangle} 

\newcommand{\lesi}{\lesssim}

\newcommand{\supp}{\operatorname{supp}}
\newcommand{\f}{\frac}

\newcommand{\om}{\omega}

\newcommand{\vc}{\infty}
\newcommand{\Rn}{\mathbb{R}^n}

\title[Harmonic analysis associated with Bessel operators]{Hardy spaces, Campanato spaces and higher order Riesz transforms associated with Bessel operators}         

\author[T. A. Bui]{The Anh Bui}
\address{School of Mathematical and Physical Sciences, Macquarie University, NSW 2109,
	Australia}
\email{the.bui@mq.edu.au}


\begin{document}
	
	\begin{abstract}
		Let $\nu = (\nu_1, \ldots, \nu_n) \in (-1/2, \vc)^n$, with $n \ge 1$, and let $\Delta_\nu$ be the multivariate Bessel operator defined by  
		\[
		\Delta_{\nu} = -\sum_{j=1}^n\left( \frac{\partial^2}{\partial x_j^2} - \frac{\nu_j^2 - 1/4}{x_j^2} \right).
		\]
		In this paper, we develop the theory of Hardy spaces and BMO-type spaces associated with the Bessel operator $\Delta_\nu$. We then study the higher-order Riesz transforms associated with $\Delta_\nu$. First, we show that these transforms are Calderón-Zygmund operators. We further prove that they are bounded on the  Hardy spaces and BMO-type spaces associated with $\Delta_\nu$.
		
	\end{abstract}
	\date{}

	\maketitle
	
	
	\section{Introduction}\label{sec: intro}
	In this paper, for $\nu\in (-1/2,\vc)^n$ we consider the multi-variate Bessel operator
	\begin{equation} \label{eq:Bessel_operator}
		\Delta_{\nu} = -\sum_{j=1}^n\left( \frac{\partial^2}{\partial x_j^2} - \frac{\nu^2 - 1/4}{x_j^2} \right).
	\end{equation}
	The operator $\Delta_{\nu}$ is a positive self-adjoint operator in $L^2((0,\infty), dx)$. The eigenfunctions of $\Delta_{\nu}$ are $\{\varphi_y\}_{y\in \Rn_+}$, where
	\begin{equation} \label{eq:eigenfunctions}
		\varphi_y(x) = \prod_{j=1}^n(y_jx_j)^{1/2} J_{\nu_j}(y_jx_j), \quad \text{and} \quad \Delta_{\nu} \varphi_y(x) = |y|^2 \varphi_y(x),
	\end{equation}
	where $J_{\alpha}(z)$ is the Bessel function of the first kind of order $\alpha$. See \cite{MS}.
	
	The $j$-th partial derivative associated with $\Delta_{\nu}$ is given by
	\[
	\delta_{\nu_j} = \frac{\partial}{\partial x_j} -\frac{1}{x_j}\Big(\nu_j + \f{1}{2}\Big).
	\]
	Then the  adjoint of $\delta_{\nu_j}$ in $L^2(\mathbb{R}^n_+)$ is
	\[
	\delta_{\nu_j}^* = -\frac{\partial}{\partial x_j} -\frac{1}{x_j}\Big(\nu_j + \f{1}{2}\Big).
	\]
	It is straightforward that
	\begin{equation}\label{eq-Delta nu}
	\Delta_\nu= \sum_{j=1}^{n} \delta_{\nu_j}^* \delta_{\nu_j}.
	\end{equation}
	The main aim of this paper is to develop the theory of Hardy spaces associated to Bessel operator $\Delta_\nu$ and investigate the boundedness of the higher order Riesz transforms associated with the Bessel operator.
	
	\bigskip
	
	\textbf{Hardy spaces associated to Bessel operator $\Delta_\nu$.} The theory of Hardy spaces associated with differential operators is a rich and active area of research in harmonic analysis, and it has attracted considerable attention. See, for example, \cite{BDT, BDK, CFYY, DY, Dziu, HLMMY, JY, PK, SY, Yan, YZ, YZ2, YYZ, YYZ2} and the references therein.
	
	Regarding the Bessel operator, for $p\in (0,1]$, we first define the Hardy space $H^p_{\Delta_\nu}(\Rn_+)$ associated with Bessel operator $\Delta_\nu$ as the completion of 
	\[
	\{f\in L^2(\Rn_+): \mathcal M_{\Delta_\nu}f:=\sup_{t>0}|e^{-t\Delta_\nu}f| \in L^p(\Rn_+)\}
	\]
	under the norm 
	\[
	\|f\|_{H^p_{\Delta_\nu}(\Rn_+)} := \|\mathcal M_{\Delta_\nu}f\|_p.
	\]
	
	In \cite{F}, Fridli introduced the atomic Hardy-type space $H^1_F(\Rn_+)$ for the case $n=1$ as follows.  
	A measurable function $a$ on $(0, \infty)$ is called an $F$-atom if it satisfies one of the following conditions:
	
	\begin{itemize}
		\item[(a)] There exists $\delta > 0$ such that
		\[
		a = \frac{1}{\delta} \chi_{(0,\delta)},
		\]
		where $\chi_{(0,\delta)}$ denotes the characteristic function of the interval $(0, \delta)$.
		
		\item[(b)] There exists a bounded interval $I \subset (0, \infty)$ such that $\operatorname{supp} a \subset I$,  
		\[
		\int_I a(x) \, dx = 0,
		\]
		and
		\[
		\|a\|_{L^\infty((0,\infty), dx)} \leq \frac{1}{|I|},
		\]
		where $|I|$ is the length of $I$.
	\end{itemize}
	
	A function $f \in L^1((0, \infty), dx)$ belongs to $H^1_F((0, \infty), dx)$ if and only if it can be expressed as  
	\[
	f(x) = \sum_{j=1}^{\infty} \alpha_j a_j(x),
	\]
	where for each $j \in \mathbb{N}$, $a_j$ is an $F$-atom and $\alpha_j \in \mathbb{C}$, satisfying  
	\[
	\sum_{j=1}^{\infty} |\alpha_j| < \infty.
	\]
	
	The norm in $H^1_F(0,\infty)$ is defined by  
	\begin{equation} \label{eq:H1F_norm}
		\|f\|_{H^1_F(0,\infty)} = \inf \sum_{j=1}^{\infty} |\alpha_j|,
	\end{equation}
	where the infimum is taken over all representations of $f$ in terms of absolutely summable sequences $\{\alpha_j\}_{j\in\mathbb{N}}$ with  
	\[
	f = \sum_{j=1}^{\infty} \alpha_j a_j, \quad a_j \text{ being an } F\text{-atom for each } j \in \mathbb{N}.
	\]
	In \cite{BDT}, it was proved that the two Hardy spaces  $H^1_{\Delta_\nu}(\mathbb R_+)$ and $H^1_{F}(\mathbb R_+)$ coincide with equivalent norms. Our first aim is to extend this result to $0<p\le 1$ and $n\ge 1$. To do this, for $x\in \Rn_+$ define
	\begin{equation}\label{eq- critical function}
		\rho(x)=\f{1}{16}\min\{x_1,\ldots,x_n\}.
	\end{equation}
	It is clear that for each $x\in (0,\vc)^n$ we have $\rho(y)\sim \rho(x)$ for $y\in B(x,\rho(x))$. Throughout the paper, we will use this frequently without giving any explanation. 
	
	\begin{defn}\label{def: rho atoms}
		Let $\nu\in (-1/2,\vc)^n$ and $\rho$ be the function as in \eqref{eq- critical function}. Let $p\in (0,1]$. A function $a$ is called a  $(p,\rho)$-atom associated to the ball $B(x_0,r)$ if
		\begin{enumerate}[{\rm (i)}]
			\item ${\rm supp}\, a\subset B(x_0,r)$;
			\item $\|a\|_{L^\vc}\leq |B(x_0,r)|^{-1/p}$;
			\item $\displaystyle \int a(x)x^\alpha dx =0$ for all multi-indices $\alpha$ with $|\alpha| \le n(1/p-1)$ if $r< \rho(x_0)$.
		\end{enumerate}
	\end{defn}
	
	If $\rho(x_0) =1$, the $(p,\rho)$ coincides with the local atoms defined in \cite{G}. Therefore, we can view a $(p,\rho)$ atom as a local atom associated to the critical function $\rho$.
	Let $p\in (0,1]$. We  say that $f=\sum_j	\lambda_ja_j$ is an atomic $(p,\rho)$-representation if
	$\{\lambda_j\}_{j=0}^\infty\in l^p$, each $a_j$ is a $(p,\rho)$-atom,
	and the sum converges in $L^2(\mathbb{R}^n_+)$. The space $H^{p}_{\rho}(\mathbb{R}^n_+)$ is then defined as the completion of
	\[
	\left\{f\in L^2(\mathbb{R}^n_+):f \ \text{has an atomic
		$(p,\rho)$-representation}\right\}
	\]
	under the norm given by
	$$
	\|f\|_{H^{p}_{\rho}(\mathbb{R}^n_+)}=\inf\Big\{\Big(\sum_j|\lambda_j|^p\Big)^{1/p}:
	f=\sum_j \lambda_ja_j \ \text{is an atomic $(p,\rho)$-representation}\Big\}.
	$$

	\begin{rem}
		\label{rem1}
		In Definition \ref{def: rho atoms}, for a $(p,\rho)$ atom $a$ associated with the ball $B(x_0,r)$, if we impose an additional  restriction  $r\le \rho(x_0)$, then the Hardy spaces defined by using these atoms are equivalent to those defined by $(p,\rho)$ atoms as in Definition \ref{def: rho atoms}, which is without the restriction on $r\le \rho(x_0)$.
	\end{rem}
	
	In the particular case when $n=p=1$,  atoms defined in Definition \ref{def: rho atoms} are a little bit different from $F$-atoms. However, it is not difficult to show that the Hardy spaces $H^1_\rho(\mathbb R_+)$ and $H^1_F(\mathbb R_+)$ are identical. 
	
	Our first main result is the following.
	\begin{thm}\label{mainthm2s}
		Let $\nu\in (-1/2,\vc)^n$ and $\gamma_\nu =  \nu_{\min}+1/2 $, where $\nu_{\min}=\min\{\nu_j: j=1,\ldots, n\}$. For $p\in (\f{n}{n+\gamma_\nu},1]$, we have
		\[
		H^{p}_{\rho}(\mathbb{R}^n_+)\equiv H^p_{\Delta_\nu}(\mathbb{R}^n_+)
		\]
		with equivalent norms.
	\end{thm}
	For  $\nu \in (-1/2,\vc)^n$, $\f{n}{n+\gamma_\nu}$ is strictly less than $1$. When $\nu_j\to \vc$, we have $\f{n}{n+\gamma_\nu}\to 0$. In general, for larger values of $\nu_j$, we have a larger range of $p$.\\
	
	As a counter part, we investigate a BMO type space, which will be proved to be a dual space of the Hardy space $H^p_\rho(\Rn_+)$. 
	
	Let $P \in \mathcal{P}_M$ be the set of all polynomials  of degree at most $M$.  For any $g \in L^1_{\rm loc} (\Rn)$ and any ball $B \subset \Rn$, we denote $P_B^M g$ the \textit{minimizing polynomial}
	of $g$ on the ball $B$ with  degree  $\leq M$, which means that $P_B^M g$ is the unique polynomial $P \in \mathcal{P}_M$
	such that,
	\begin{align} \label{minipoly}
		\int_B [g(x)-P(x)]x^\alpha dx  =0 \quad \mbox{for every } |\alpha|\le M.
	\end{align}
	It is known that if $g$ is locally integrable, then $P^M_B g$ uniquely exists (see \cite{JTW}). 
	We define the local Campanato spaces associated to critical functions as follows.
	\begin{defn}\label{defn1}
		Let $\nu\in (-1/2,\vc)^n$. Let $\rho$ be the critical function as in \eqref{eq- critical function}. 
		Let $s \geq 0$, $1 \leq q \leq \infty$ and $M \in \mathbb{N}$.
		The local Campanato space $BMO^{s,M}_{\rho}(\Rn_+)$ associated to $\rho$ is defined to
		be the space of all locally $L^1$ functions $f$ on $\Rn_+$ such that
		\begin{equation*}
			\begin{split}
				\|f\|_{BMO^{s, M}_{\rho}(\Rn_+)}&:=\sup_{\substack{B: \  {\rm balls}\\ r_B < \rho (x_B)}}    \frac{1}{|B|^{s/n}}
				 \Big(\f{1}{|B|}\int_B|f(x)-P^M_Bf(x)|^2  dx\Big)^{1/2}  \\
				& \quad\quad +\sup_{\substack{B: \ {\rm balls}\\ r_B \geq \rho (x_B)} }   \frac{1}{|B|^{s/n}}
				 \Big(\f{1}{|B|}\int_B|f(x)|^2 dx  \Big)^{1/2} <\infty.
			\end{split}
		\end{equation*}
	Here, $x_B$ and $r_B$ denote the center and radius of the ball $B$, respectively.
	\end{defn}
	
In Definition~\ref{defn1}, when $r_B < \rho(x_B)$, we have $B \subset \mathbb{R}^n_+$. In this case,   $P_B^M f$ should be understood as $P_B^M \tilde{f}$, where $\tilde{f}$ denotes the zero extension of $f$ to $\mathbb{R}^n$. However, for convenience, we continue to write $P_B^M f$ without risk of confusion.

The we have the following result regarding the duality of the Hardy space $H^p_{\Delta_\nu}(\mathbb R^n_+)$.
\begin{thm} \label{mainthm-dual}  Let $\nu\in (-1/2,\vc)^n$ and $\gamma_\nu =  \nu_{\min}+1/2$, where $\nu_{\min}=\min\{\nu_j: j=1,\ldots, n\}$. Let $\rho$ be as in \eqref{eq- critical function}. Then for  $p \in (\f{n}{n+\gamma_\nu},1]$, we have
	\begin{align*}
		\big(H^p_{\Delta_\nu} (\mathbb{R}^n_+)\big)^\ast = BMO^{s,M}_{\rho } (\mathbb{R}^n_+),
		\quad \mbox{where } s: = n (1/p-1)
	\end{align*}
	for all $M\in \mathbb N$ with $M\ge \lfloor s\rfloor$, where $s$ is the greatest integer smaller than or equal to $s$.
\end{thm}	

Due to Theorem \ref{mainthm-dual}, for $s\ge 0$ we define $BMO^{s}_{\rho } (\mathbb{R}^n_+)$ as any space $BMO^{s,M}_{\rho } (\mathbb{R}^n_+)$ with $M\in \mathbb N$ with $M\ge \lfloor s\rfloor$.

\bigskip

\noindent \textbf{Riesz transforms associated to Bessel operators.} The study of Riesz transforms associated with differential operators is a central topic in harmonic analysis and has been extensively investigated. See, for example, \cite{BDT, Betancor1, Betancor2, BD, Muc, MS, NS, NS2, MST1, MST2, Th} and the references therein. Let $k=(k_1,\ldots, k_n)\in \mathbb N^n$ we consider the higher order Riesz transform $\delta_\nu^k \Delta_\nu^{-|k|/2}$, where $\delta_\nu^k=\delta_{\nu_n}^{k_n}\ldots \delta_{\nu_1}^{k_1}$. See Section 4 for the definition of $\Delta_\nu^{-|k|/2}$. Regarding the Riesz transform associated to Bessel operators, in the $1$-dimensional case $n=1$, it was proved in \cite{Betancor2} that the Riesz transform $\delta_\nu \Delta_\nu^{-1/2}$ is a Calder\'on-Zygmund operator. In \cite{Betancor1} (also for the case $n=1$), a different version of the higher order Riesz transform in the Bessel setting was investigated. Note that  the higher order Riesz transform \cite{Betancor2} are defined through the higher order  Riesz transform associated to Laplacian of Bessel-type operator 
\[
-\frac{\partial^2}{\partial x^2} - \frac{\nu+1}{x} \frac{\partial}{\partial x},
\]
and it is definitely not the higher Riesz transform $\delta_\nu^k \Delta_\nu^{-|k|/2}$ as expected due to some technical reasons.

Our first main result in this section is to show that the operator $\delta_\nu^k \Delta_\nu^{-|k|/2}$	is a Calder\'on-Zygmund operator. 
	
\begin{thm}\label{thm-Riesz transform} Let $\nu\in (-1/2,\vc)^n$, $\nu_{\min}=\min\{\nu_j: j=1,\ldots, n\}$ and $k=(k_1,\ldots, k_n)\in \mathbb N^n$ be a multi-index. Then the Riesz transform $\delta^k_\nu \Delta_\nu^{-|k|/2}$ is a Calder\'on-Zygmund operator. That is,	$\delta^k_\nu \Delta_\nu^{-|k|/2}$ is bounded on $L^2(\Rn_+)$ and its kernel  $\delta^k_\nu \Delta_\nu^{-|k|/2}(x,y)$ satisfies the following estimates:
	\[
	|\delta^k_\nu \Delta_\nu^{-|k|/2}(x,y)|\lesi  \f{1}{|x-y|^n} , \ \ \ x\ne y
	\]
	and
	\[
	\begin{aligned}
		| \delta^k_\nu \Delta_\nu^{-|k|/2}(x,y)-\delta^k_\nu \Delta_\nu^{-|k|/2}(x,y')|&+| \delta^k_\nu \Delta_\nu^{-|k|/2}(y,x)-\delta^k_\nu \Delta_\nu^{-|k|/2}(y',x)|
		&\lesi \Big(\f{|y-y'|}{x-y}\Big)^{\nu_{\min}+1/2}\f{1}{|x-y|^n},
	\end{aligned}
	\]
	whenever $|y-y'|\le |x-y|/2$.
\end{thm}	
\begin{thm}\label{thm- boundedness on Hardy and BMO}
	Let $\nu\in (-1/2, \vc)^n$, $\gamma_\nu=\nu_{\min}+1/2$ and $k\in \mathbb N^n$, where $\nu_{\min}=\min\{\nu_j: j=1,\ldots, n\}$. Then for  $\f{n}{n+\gamma_\nu}<p\le 1$ and $s=n(1/p-1)$, we have
	\begin{enumerate}[{\rm (i)}]
		\item the Riesz transform $\delta^k_\nu \Delta_\nu^{-|k|/2}$ is bounded on $H^p_{\rho}(\mathbb{R}^n_+)$;
		\item the Riesz transform $\delta^k_\nu \Delta_\nu^{-|k|/2}$ is bounded on $BMO^s_{\rho}(\mathbb{R}^n_+)$.
	\end{enumerate} 
\end{thm}
	
	\bigskip
Although our approach is closely related to that in \cite{BD}, several new ideas and improvements are necessary due to fundamental differences in our setting. The techniques in \cite{BD} heavily rely on the discrete eigenvalues of the Laguerre operator and the fact that its eigenvectors form an orthonormal basis for $L^2(\Rn_+)$. However, these properties do not hold in our case, requiring alternative methods. For instance, in estimating the heat kernels, we must develop a direct approach rather than leveraging the special properties of the derivative operator $\delta_\nu$ acting on the semigroups, as was done in \cite{BD}. Furthermore, establishing the boundedness of the Riesz transform is significantly more challenging, as we cannot rely on specific structural properties of the eigenvalues and eigenfunctions that were instrumental in \cite{BD}. These distinctions necessitate a refined analytical framework to address the difficulties that arise in our setting. \\

Note that the restriction $\nu\in (-1/2,\vc)^n$ is essential to guarantee that the higher-order Riesz transforms are Calder\'on-Zygmund. The more general case $\nu\in (-1,\vc)^n$ will be investigated in the forthcoming paper \cite{B2}.\\
	
	\bigskip
	
	Throughout the paper, we always use $C$ and $c$ to denote positive constants that are independent of the main parameters involved but whose values may differ from line to line. We will write $A\lesi B$ if there is a universal constant $C$ so that $A\leq CB$ and $A\sim B$ if $A\lesi B$ and $B\lesi A$. For $a \in \mathbb{R}$, we denote the integer part of $a$ by $\lfloor a\rfloor$.  For a given ball $B$, unless specified otherwise, we shall use $x_B$ to denote the center and $r_B$ for the radius of the ball. We also denote $a\vee b =\max\{a,b\}$ and $a\wedge b=\min\{a,b\}$.
	
	In the whole paper, we will often use the following inequality without any explanation $e^{-x}\le c(\alpha) x^{-\alpha}$ for any $\alpha>0$ and $x>0$.

	\section{Some   kernel estimates}
	
	This section is devoted to establishing some kernel estimates related to the heat kernel of $\Delta_\nu$. These estimates play an essential role in proving our main results. We begin by providing an explicit formula for the heat kernel of $\Delta_\nu$.
	
	Let $\nu \in (-1,\vc)^n$. For each $j=1,\ldots, n$, denote
	\[
	\Delta_{\nu_j} :=  -\frac{\partial^2}{\partial x_j^2} + \frac{\nu_j^2 - 1/4}{x_j^2} 
	\] 
	on $C_c^\infty(\mathbb{R}_+)$ as the natural domain. It is easy to see that 
	\[
	\Delta_\nu =\sum_{j=1}^n \Delta_{\nu_j}.
	\]

	Let $p_t^\nu(x,y)$ be the kernel of $e^{-t\Delta_\nu}$ and let $p_t^{\nu_j}(x_j,y_j)$ be the kernel of $e^{-t\Delta_{\nu_j}}$ for each $j=1,\ldots, n$. Then we have
	\begin{equation}\label{eq- prod ptnu}
		p_t^\nu(x,y)=\prod_{j=1}^n p_t^{\nu_j}(x_j,y_j).
	\end{equation}
	For $\nu_j\ge -1/2$, $j=1,\ldots, n,$, the kernel of $e^{-t\Delta_{\nu_j}}$ is given by
	\begin{equation}
		\label{eq1-ptxy}
		p_t^{\nu_j}(x_j,y_j)=\f{ (x_jy_j)^{1/2}}{2t}\exp\Big(-\f{x_j^2+y_j^2}{4t}\Big)I_{\nu_j}\Big(\f{x_jy_j}{2t}\Big),
	\end{equation}
	where  $I_\alpha$ is the usual Bessel funtions of an imaginary argument defined by
	\[
	I_\alpha(z)=\sum_{k=0}^\vc \f{\Big(\f{z}{2}\Big)^{\alpha+2k}}{k! \Gamma(\alpha+k+1)}, \ \ \ \ \alpha >-1.
	\]
	See for example \cite{Dziu, NS}.

	Note that for each $j=1,\ldots, n$, we can rewrite the kernel $p_t^{\nu_j}(x_j,y_j)$ as follows
	\begin{equation}
		\label{eq2-ptxy}
		\begin{aligned}
			p_t^{\nu_j}(x_j,y_j)=\f{1}{\sqrt {2t}}\Big(\f{x_jy_j}{2t}\Big)^{\nu_j+1/2} \exp\Big(-\f{x_j^2+y_j^2}{4t}\Big)\Big(\f{x_jy_j}{2t}\Big)^{-\nu_j}I_{\nu_j}\Big(\f{x_jy_j}{2t}\Big).
		\end{aligned}
	\end{equation}

	The following  properties of the Bessel function $I_\alpha$  with $\alpha>-1/2$ are well-known and are taken from \cite{L}:
	\begin{equation}
		\label{eq1-Inu}
		I_\alpha(z)\sim z^\alpha, \ \ \ 0<z\le 1,
	\end{equation}
	\begin{equation}
		\label{eq2-Inu}
		I_\alpha(z)= \f{e^z}{\sqrt{2\pi z}}+S_\alpha(z),
	\end{equation}
	where
	\begin{equation}
		\label{eq3-Inu}
		|S_\alpha(z)|\le  Ce^zz^{-3/2}, \ \ z\ge 1,
	\end{equation}
	\begin{equation}
		\label{eq4-Inu}
		\f{d}{dz}(z^{-\alpha}I_\alpha(z))=z^{-\alpha}I_{\alpha+1}(z),
	\end{equation}
	
	\begin{equation}
		\label{eq5s-Inu}
		0< I_\alpha(z)-I_{\alpha+1}(z)<2(\alpha+1)\f{I_{\alpha+1}(z)}{z}, \ \ \ z>0.
	\end{equation}
	and
	\begin{equation}
		\label{eq6s-Inu}
		0< I_\alpha(z)-I_{\alpha+2}(z)=\f{2(\alpha+1)}{z}I_{\alpha+1}(z), \ \ \ z>0.
	\end{equation}
	In the case $n=1$, from \eqref{eq5s-Inu}, \eqref{eq6s-Inu} and \eqref{eq1-ptxy}, we have
	\begin{equation}
		\label{eq5-Inu}
		0< p_t^\alpha(x,y)-p_t^{\alpha+1}(x,y)<2(\alpha+1)\f{t}{xy}p_t^{\alpha+1}(x,y), \ \ \ z>0.
	\end{equation}
	and
	\begin{equation}
		\label{eq6-Inu}
		0< p_t^\alpha(x,y)-p_t^{\alpha+2}(x,y)=\f{2(\alpha+1)t}{xy}p_t^{\alpha+1}(x,y), \ \ \ z>0.
	\end{equation}
	
	\begin{rem}
		When $n=1$, from \eqref{eq5-Inu}  we imply directly that for $\nu>-1$ we have $p_t^{\nu+1}(x,y)\le p_t^{\nu}(x,y)$ for all $t>0$ and $x,y>0$. We will use this inequality frequently without any further explanation.
	\end{rem}
	Before coming the the kernel estimates, we need the following simple identities
		\begin{equation}\label{eq-formula for delta k xf}
		\delta^{k}_\nu \Big[xf(x)\Big]= k\delta^{k-1}_\nu  f(x) + x\delta^{k}_\nu  f(x)
	\end{equation}
	and
	\begin{equation}\label{eq- del nu and del nu + 1}
		\delta_\nu^{k}=\Big(\delta_{\nu+1}+\f{1}{x}\Big)^{k} =\delta_{\nu+1}^k +\f{k}{x}\delta^{k-1}_{\nu+1}.
	\end{equation}
	
	\subsection{The case $n=1$}	We first write 
	\[
	p_t^\nu(x,y)=\f{ 1}{\sqrt{2t}}\Big(\f{xy}{2t}\Big)^{\nu+1/2}\exp\Big(-\f{x^2+y^2}{4t}\Big)\Big(\f{xy}{2t}\Big)^{-\nu}I_{\nu}\Big(\f{xy}{2t}\Big).
	\]
	Hence, using \eqref{eq4-Inu}, it can be verified that
	\[
	\partial_x p_t^\nu(x,y) = \f{(\nu+1/2)}{x} p_t^\nu(x,y) -\f{x}{2t}p_t^\nu(x,y) +p_t^{\nu+1}(x,y),
	\]
	which implies
			\begin{align}
			\delta_\nu p_t^\nu(x,y) &= -\f{x}{2t}p_t^\nu(x,y) +\f{y}{2t}p_t^{\nu+1}(x,y) \label{eq1-delta p}\\
			&= -\f{x}{2t}\big[p_t^\nu(x,y)-p_t^{\nu+1}(x,y)\big] +\f{y-x}{2t}p_t^{\nu+1}(x,y). \label{eq2-delta p}
		\end{align}

	\begin{thm}\label{thm-heat kernel}
		Let $\nu>-1/2$. Then
		\[
		|p_t^\nu(x,y)|\lesi \f{1}{\sqrt t}\exp\Big(-\f{|x-y|^2}{ct}\Big)\Big(1+\f{\sqrt t}{x}\Big)^{-\nu-1/2}\Big(1+\f{\sqrt t}{y}\Big)^{-\nu-1/2}
		\]
		for all $t>0$ and $x,y\in (0,\vc)$.
	\end{thm}
	\begin{proof}
		\textbf{Case 1:  $xy\le 2t$.} Using \eqref{eq5-Inu},  \eqref{eq1-ptxy} and \eqref{eq1-Inu} we have 
		\[
		\begin{aligned}
			p_t^\nu(x,y) 
			&\lesi  \f{1}{\sqrt t}\Big(\f{xy}{t}\Big)^{\nu+1/2}\exp\Big(-\f{x^2+y^2}{4t}\Big)\\
			&\lesi \f{1}{\sqrt t}\exp\Big(-\f{|x-y|^2}{ct}\Big)\Big(1+\f{\sqrt t}{x}\Big)^{-\nu-1/2}\Big(1+\f{\sqrt t}{y}\Big)^{-\nu-1/2}.
		\end{aligned}
		\]
		
		\textbf{Case 2:  $xy> 2t$.} Using \eqref{eq5-Inu},  \eqref{eq2-Inu} and \eqref{eq3-Inu} we have
		\[
		\begin{aligned}
			p_t^\nu(x,y)&\lesi  \f{1}{\sqrt t}\exp\Big(-\f{|x-y|^2}{ct}\Big).
		\end{aligned}
		\] 
		
		If $x>2y$, then $|x-y|\sim x$. Consequently,
		\[
		\begin{aligned}
			p_t^\nu(x,y)&\lesi  \f{1}{\sqrt t}\exp\Big(-\f{|x-y|^2}{ct}\Big)\\
			&\lesi  \f{1}{\sqrt t}\Big(\f{\sqrt t}{x}\Big)^{2(\nu+1/2)}\exp\Big(-\f{|x-y|^2}{ct}\Big)\\
			&\lesi \f{1}{\sqrt t}\exp\Big(-\f{|x-y|^2}{ct}\Big)\Big(1+\f{\sqrt t}{x}\Big)^{-\nu-1/2}\Big(1+\f{\sqrt t}{y}\Big)^{-\nu-1/2}.
		\end{aligned}
		\]
		Similarly, if $y>2x$, we also have
		\[
		p_t^\nu(x,y)\lesi \f{1}{\sqrt t}\exp\Big(-\f{|x-y|^2}{ct}\Big)\Big(1+\f{\sqrt t}{x}\Big)^{-\nu-1/2}\Big(1+\f{\sqrt t}{y}\Big)^{-\nu-1/2}.
		\]
		In the remaining case if $y/2\le x\le 2y$, then $x\sim y\gtrsim \sqrt t$. It follows that 
		$$
		\Big(1+\f{\sqrt t}{x}\Big)^{-\nu-1/2}\Big(1+\f{\sqrt t}{y}\Big)^{-\nu-1/2}\sim 1.
		$$
		Hence,
		\[
		p_t^\nu(x,y)\lesi \f{1}{\sqrt t}\exp\Big(-\f{|x-y|^2}{ct}\Big)\Big(1+\f{\sqrt t}{x}\Big)^{-\nu-1/2}\Big(1+\f{\sqrt t}{y}\Big)^{-\nu-1/2}.
		\]
		This completes our proof.
	\end{proof}
	
	\begin{prop}\label{prop- delta k p nu first region}
		For $\ell\in \mathbb N$,
		\[
		| \delta^\ell_\nu p_t^\nu(x,y)| \lesi_{\nu,\ell} \f{1}{t^{(\ell+1)/2}}\exp\Big(-\f{|x-y|^2}{ct}\Big)\Big(1+\f{\sqrt t}{x}\Big)^{-\nu-1/2}\Big(1+\f{\sqrt t}{y}\Big)^{-\nu-1/2}
		\]
		for all $\nu>-1/2$, $t>0$ and $x<\sqrt t$ OR $x>2y$ OR $x<y/2$.
	\end{prop}
	\begin{proof}
		We will prove by induction.
		
		The estimate holds true for $\ell=0$ due to Theorem \ref{thm-heat kernel}. Assume that the estimate holds true for $\ell=0,1,\ldots, k$ for some $k\ge 0$, i.e., for $\ell=0,1,\ldots, k$,
		\begin{equation}
			\label{eq-inductive hypo 1st region}
			| \delta^\ell_\nu p_t^\nu(x,y)| \lesi \f{1}{t^{(\ell+1)/2}}\exp\Big(-\f{|x-y|^2}{ct}\Big)\Big(1+\f{\sqrt t}{x}\Big)^{-\nu-1/2}\Big(1+\f{\sqrt t}{y}\Big)^{-\nu-1/2}
		\end{equation}
			for all $\nu>-1/2$, $t>0$ and $x<\sqrt t$ or $x>2y$ or $x<y/2$.
			
		We need to prove the estimate for $\ell = k +1$  and $\nu>-1/2$. From \eqref{eq1-delta p} we have
		\[
		|\delta^{k+1}_\nu p_t^\nu(x,y)|\lesi  \Big|\delta_\nu^k\Big[\f{x}{2t}p_t^\nu(x,y)\Big]\Big| + \f{y}{2t}  |\delta_\nu^k p_t^{\nu+1}(x,y)|:=E_{1}+E_{2}.
		\]
		Applying \eqref{eq-formula for delta k xf},
		\[
		\begin{aligned}
			E_{1}&\lesi \f{x}{t}|\delta_\nu^kp_t^\nu(x,y)| + \f{1}{t}|\delta_\nu^{k-1}p_t^\nu(x,y)|.
		\end{aligned}
		\]
		By \eqref{eq-inductive hypo 1st region},
		\[
		\f{1}{t}|\delta_\nu^{k-1}p_t^\nu(x,y)|\lesi \f{1}{t^{(k+2)/2}}\exp\Big(-\f{|x-y|^2}{ct}\Big)\Big(1+\f{\sqrt t}{x}\Big)^{-\nu-1/2}\Big(1+\f{\sqrt t}{y}\Big)^{-\nu-1/2}
		\]
		for all $t>0$ and $x<\sqrt t$ or $x>2y$ or $x<y/2$.
		
		Similarly, if $x<\sqrt t$,
		\[
		\f{x}{t}|\delta_\nu^kp_t^\nu(x,y)|\lesi \f{1}{t^{(k+2)/2}}\exp\Big(-\f{|x-y|^2}{ct}\Big)\Big(1+\f{\sqrt t}{x}\Big)^{-\nu-1/2}\Big(1+\f{\sqrt t}{y}\Big)^{-\nu-1/2}.		
		\]
		If $x\ge 2y$ or $x\le y/2$, then $x \lesi |x-y|$. Therefore, by \eqref{eq-inductive hypo 1st region},
		\[
		\begin{aligned}
			\f{x}{t}|\delta_\nu^kp_t^\nu(x,y)|&\lesi \f{x}{t} \f{1}{t^{(k+1)/2}}\exp\Big(-\f{|x-y|^2}{ct}\Big)\Big(1+\f{\sqrt t}{x}\Big)^{-\nu-1/2}\Big(1+\f{\sqrt t}{y}\Big)^{-\nu-1/2}\\
			&\lesi \f{x}{t} \f{\sqrt t}{|x-y|}\f{1}{t^{(k+1)/2}}\exp\Big(-\f{|x-y|^2}{2ct}\Big)\Big(1+\f{\sqrt t}{x}\Big)^{-\nu-1/2}\Big(1+\f{\sqrt t}{y}\Big)^{-\nu-1/2}\\
			&\lesi \f{1}{t^{(k+2)/2}}\exp\Big(-\f{|x-y|^2}{ct}\Big)\Big(1+\f{\sqrt t}{x}\Big)^{-\nu-1/2}\Big(1+\f{\sqrt t}{y}\Big)^{-\nu-1/2},
		\end{aligned}
		\]
		as long as $x\ge 2y$ or $x\le y/2$.
		
		Hence, 
		\[
		E_1\lesi \f{1}{t^{(k+2)/2}}\exp\Big(-\f{|x-y|^2}{ct}\Big)\Big(1+\f{\sqrt t}{x}\Big)^{-\nu-1/2}\Big(1+\f{\sqrt t}{y}\Big)^{-\nu-1/2}
		\]
				for all $t>0$ and $x<\sqrt t$ or $x>2y$ or $x<y/2$.
				
		For the term $E_2$, using \ref{eq- del nu and del nu + 1} to obtain
		\[
		\begin{aligned}
			E_2 &\lesi \f{y}{t}|\delta_{\nu+1}^k p_t^{\nu+1}(x,y)| + \f{y}{tx}|\delta_{\nu+1}^{k-1} p_t^{\nu+1}(x,y)|.
		\end{aligned}
		\]
		Similarly to the estimate of $E_1$, we have
		\[
		\begin{aligned}
			\f{y}{t}|\delta_{\nu+1}^k p_t^{\nu+1}(x,y)|&\lesi \f{|y-x|}{t}|\delta_{\nu+1}^k p_t^{\nu+1}(x,y)|+\f{x}{t}|\delta_{\nu+1}^k p_t^{\nu+1}(x,y)|\\
			&\lesi \f{1}{t^{(k+2)/2}}\exp\Big(-\f{|x-y|^2}{ct}\Big)\Big(1+\f{\sqrt t}{x}\Big)^{-\nu-1/2}\Big(1+\f{\sqrt t}{y}\Big)^{-\nu-1/2}
		\end{aligned}
		\]
					for all $t>0$ and $x<\sqrt t$ or $x>2y$ or $x<y/2$.
					
		For the remaining term,
		\[
		\begin{aligned}
			\f{y}{tx}|\delta_{\nu+1}^{k-1} p_t^{\nu+1}(x,y)|&\lesi \f{|y-x|}{tx}|\delta_{\nu+1}^{k-1} p_t^{\nu+1}(x,y)|+ \f{1}{t}|\delta_{\nu+1}^{k-1} p_t^{\nu+1}(x,y)|.
		\end{aligned}
		\]
		By \eqref{eq-inductive hypo 1st region} we have
		\[
		\f{1}{t}|\delta_{\nu+1}^{k-1} p_t^{\nu+1}(x,y)|\lesi \f{1}{t^{(k+2)/2}}\exp\Big(-\f{|x-y|^2}{ct}\Big)\Big(1+\f{\sqrt t}{x}\Big)^{-\nu-1/2}\Big(1+\f{\sqrt t}{y}\Big)^{-\nu-1/2}
		\]
		and
		\[
		\begin{aligned}
			\f{|y-x|}{tx}|\delta_{\nu+1}^{k-1} p_t^{\nu+1}(x,y)|&\lesi \f{|y-x|}{tx}\f{1}{t^{k/2}}\exp\Big(-\f{|x-y|^2}{ct}\Big)\Big(1+\f{\sqrt t}{x}\Big)^{-\nu-3/2}\Big(1+\f{\sqrt t}{y}\Big)^{-\nu-3/2}\\
			&\lesi \f{1}{\sqrt t x}\Big(1+\f{\sqrt t}{x}\Big)^{-1}\f{1}{t^{k/2}}\exp\Big(-\f{|x-y|^2}{2ct}\Big)\Big(1+\f{\sqrt t}{x}\Big)^{-\nu-1/2}\Big(1+\f{\sqrt t}{y}\Big)^{-\nu-1/2}\\
			&\lesi  \f{1}{t^{(k+2)/2}}\exp\Big(-\f{|x-y|^2}{2ct}\Big)\Big(1+\f{\sqrt t}{x}\Big)^{-\nu-1/2}\Big(1+\f{\sqrt t}{y}\Big)^{-\nu-1/2}
		\end{aligned}
		\]
							for all $t>0$ and $x<\sqrt t$ or $x>2y$ or $x<y/2$.

		This completes our proof.
	\end{proof}

	\begin{prop}\label{prop-delta k pt  2nd region}
For $\ell\in \mathbb N\backslash\{0\}$, we have
	\[
	| \delta^\ell_\nu p_t^\nu(x,y)|+ \f{x}{t}|\delta^{\ell-1}_\nu[p_t^\nu(x,y)-p_t^{\nu+1}(x,y)]|\lesi_{\nu,\ell} \f{1}{t^{(\ell+1)/2}}\exp\Big(-\f{|x-y|^2}{ct}\Big)\Big(1+\f{\sqrt t}{x}\Big)^{-\nu-1/2}\Big(1+\f{\sqrt t}{y}\Big)^{-\nu-1/2}
	\]
	for all $\nu>-1/2$, $t>0$ and $x\sim y\gtrsim \sqrt t$.
\end{prop}
\begin{proof}
	We will prove the proposition by induction. 
	
	\noindent $\bullet$ We first prove the estimate for $\ell=1$. Obviously, by \eqref{eq5-Inu}, \eqref{eq1-ptxy} and Theorem \ref{thm-heat kernel},
	\[
	\begin{aligned}
		\f{x}{t}| p_t^\nu(x,y)-p_t^{\nu+1}(x,y)| &\lesi \f{x}{t}\f{2t}{xy}p_t^{\nu+1}(x,y)\\
		&\lesi \f{1}{y} \f{1}{\sqrt t}\exp\Big(-\f{|x-y|^2}{ct}\Big)\Big(1+\f{\sqrt t}{x}\Big)^{-\nu-1/2}\Big(1+\f{\sqrt t}{y}\Big)^{-\nu-1/2}\\
		&\lesi  \f{1}{t}\exp\Big(-\f{|x-y|^2}{ct}\Big)\Big(1+\f{\sqrt t}{x}\Big)^{-\nu-1/2}\Big(1+\f{\sqrt t}{y}\Big)^{-\nu-1/2}\\
	\end{aligned}
	\]
	for all $\nu>-1/2$, $t>0$ and $x\sim y\gtrsim \sqrt t$.
	
	We now estimate  $\delta_\nu p_t^\nu(x,y)$. For  $x\sim y\gtrsim\sqrt t$, using  \eqref{eq2-delta p}, 
	\[
	\delta_\nu p_t^\nu(x,y)  = -\f{x}{2t}\big[p_t^\nu(x,y)-p_t^{\nu+1}(x,y)\big] +\f{y-x}{2t}p_t^{\nu+1}(x,y).
	\]
	In this case, $xy>t$. Hence, applying \eqref{eq5-Inu}, \eqref{eq1-ptxy} and Theorem \ref{thm-heat kernel} we obtain
	\[
	\begin{aligned}
		|\sqrt t\delta_\nu p_t^\nu(x,y)|&\lesi \f{\sqrt t}{y}p_t^{\nu+1}(x,y) + \f{|y-x|}{\sqrt t}p_t^{\nu+1}(x,y)\\
		&\lesi \f{1}{\sqrt t}\exp\Big(-\f{|x-y|^2}{ct}\Big)\Big(1+\f{\sqrt t}{x}\Big)^{-\nu-1/2}\Big(1+\f{\sqrt t}{y}\Big)^{-\nu-1/2}
	\end{aligned}
	\] 
		for all $t>0$ and $x\sim y\gtrsim \sqrt t$.
	\bigskip
	
	$\bullet$ Assume that the estimate is true for $\ell=1,\ldots, k$ for some $k\ge 1$, i.e., for $\ell=1,\ldots, k$ we have 	\begin{equation}\label{eq- inductive hypo delta k pt nu}
	| \delta^\ell_\nu p_t^\nu(x,y)|+ \f{x}{t}|\delta^{\ell-1}_\nu[p_t^\nu(x,y)-p_t^{\nu+1}(x,y)]|\lesi \f{1}{t^{(\ell+1)/2}}\exp\Big(-\f{|x-y|^2}{ct}\Big)\Big(1+\f{\sqrt t}{x}\Big)^{-\nu-1/2}\Big(1+\f{\sqrt t}{y}\Big)^{-\nu-1/2}
	\end{equation}
	for all $\nu>-1/2$, $t>0$ and $x\sim y\gtrsim \sqrt t$.
	
	 We need to prove the estimate for $\ell=k+1$ and $\nu>-1/2$. We first have
	\[
	\begin{aligned}
		\delta_\nu^k[p_t^\nu(x,y)-p_t^{\nu+1}(x,y)]&=\delta_\nu^{k-1}\big[\delta_\nu p_t^\nu(x,y)-\delta_{\nu+1}p_t^{\nu+1}(x,y)-\f{1}{x}p_t^{\nu+1}(x,y)]\big]\\
		&=\delta_\nu^{k-1}\big[\delta_\nu p_t^\nu(x,y)-\delta_{\nu+1}p_t^{\nu+1}(x,y)\big] -\delta_\nu^{k-1}\Big[\f{1}{x}p_t^{\nu+1}(x,y)\Big],
	\end{aligned}
	\]
	which implies
		\[
	\begin{aligned}
		\f{x}{t}|\delta_\nu^k[p_t^\nu(x,y)-p_t^{\nu+1}(x,y)]|&\lesi \f{x}{t}\Big|\delta_\nu^{k-1}\big[\delta_\nu p_t^\nu(x,y)-\delta_{\nu+1}p_t^{\nu+1}(x,y)\big]\Big| +\f{x}{t}\Big|\delta_\nu^{k-1}\Big[\f{1}{x}p_t^{\nu+1}(x,y)\Big]\Big|\\
		&=E_1+E_2.
	\end{aligned}
	\]
	Using \eqref{eq- del nu and del nu + 1},
	\[
	\begin{aligned}
		E_2&\lesi \f{x}{t}\Big|\delta_{\nu+1}^{k-1}\Big[\f{1}{x}p_t^{\nu+1}(x,y)\Big]\Big| +\f{1}{t}\Big|\delta_{\nu+1}^{k-2}\Big[\f{1}{x}p_t^{\nu+1}(x,y)\Big]\Big|\\
		&\lesi \f{x}{t}\sum_{j=1}^{k-1} \f{1}{x^j} |\delta_{\nu+1}^{k-j}p_t^{\nu+1}(x,y)|+\f{1}{t}\sum_{j=1}^{k-2} \f{1}{x^j} |\delta_{\nu+1}^{k-j}p_t^{\nu+1}(x,y)|.
	\end{aligned}
	\]	
	Using \eqref{eq- inductive hypo delta k pt nu} we obtain
	\[
	E_2\lesi \f{1}{t^{(k+2)/2}}\exp\Big(-\f{|x-y|^2}{ct}\Big)\Big(1+\f{\sqrt t}{x}\Big)^{-\nu-1/2}\Big(1+\f{\sqrt t}{y}\Big)^{-\nu-1/2}
	\]
	for $t>0$ and  $x\sim y\gtrsim \sqrt t$.
	
	We now take care of $E_1$. Using \eqref{eq2-delta p} and \eqref{eq6-Inu},
	\[
	\begin{aligned}
		\delta_\nu p_t^\nu(x,y)-\delta_{\nu+1}p_t^{\nu+1}(x,y)&=\f{x}{2t}[p_t^\nu(x,y)-p_t^{\nu+2}(x,y)] +\f{y-x}{2t}[p_t^{\nu+1}(x,y)-p_t^{\nu+2}(x,y)]\\
		&=\f{2(\nu+1)}{y}p_t^{\nu+1}(x,y) +\f{y-x}{2t}[p_t^{\nu+1}(x,y)-p_t^{\nu+2}(x,y)],
	\end{aligned}
	\]
	which implies that 
	\[
	\begin{aligned}
		E_1
		&\lesi \f{x}{ty}|\delta_\nu^{k-1}p_t^{\nu+1}(x,y)| +\f{x}{t}\Big|\delta_\nu^{k-1}\Big[\f{y-x}{2t}(p_t^{\nu+1}(x,y)-p_t^{\nu+2}(x,y))\Big]\Big|\\
		&=E_{11}+E_{12}.
	\end{aligned}
	\]
	By \eqref{eq- del nu and del nu + 1} and \eqref{eq- inductive hypo delta k pt nu}, we have
	\[
	\begin{aligned}
		E_{11}&\lesi  \f{x}{ty}|\delta_{\nu+1}^{k-1}p_t^{\nu+1}(x,y)|+\f{1}{ty}|\delta_{\nu+1}^{k-2}p_t^{\nu+1}(x,y)|\\
		&\lesi \f{1}{t^{(k+2)/2}}\exp\Big(-\f{|x-y|^2}{ct}\Big)\Big(1+\f{\sqrt t}{x}\Big)^{-\nu-1/2}\Big(1+\f{\sqrt t}{y}\Big)^{-\nu-1/2}
	\end{aligned}
	\]
	for $t>0$ and $x\sim y\gtrsim \sqrt t$.
	
	By \eqref{eq- del nu and del nu + 1},
	\[
	\begin{aligned}
		|E_{12}| &\lesi \f{x}{t}\Big|\delta_{\nu+1}^{k-1}\Big[\f{y-x}{2t}(p_t^{\nu+1}(x,y)-p_t^{\nu+2}(x,y))\Big]\Big|+\f{1}{t}\Big|\delta_{\nu+1}^{k-2}\Big[\f{y-x}{2t}(p_t^{\nu+1}(x,y)-p_t^{\nu+2}(x,y))\Big]\Big|\\
		&=:E_{121}+E_{122}.
	\end{aligned}
	\]
	Applying \eqref{eq-formula for delta k xf} and \eqref{eq- inductive hypo delta k pt nu},
	\[
	\begin{aligned}
		E_{121}&\lesi \f{x}{t}\f{|y-x|}{2t}\Big|\delta_{\nu+1}^{k-1}(p_t^{\nu+1}(x,y)-p_t^{\nu+2}(x,y))\Big|+\f{x}{t^2} \Big|\delta_{\nu+1}^{k-2}(p_t^{\nu+1}(x,y)-p_t^{\nu+2}(x,y))\Big|\\
		&\lesi \f{1}{t^{(k+2)/2}}\exp\Big(-\f{|x-y|^2}{ct}\Big)\Big(1+\f{\sqrt t}{x}\Big)^{-\nu-1/2}\Big(1+\f{\sqrt t}{y}\Big)^{-\nu-1/2}
		\end{aligned}
		\]
		for $t>0$ and $x\sim y\gtrsim \sqrt t$.
		
		Similarly,
		\[
		\begin{aligned}
			E_{122}		&\lesi \f{1}{t^{(k+2)/2}}\exp\Big(-\f{|x-y|^2}{ct}\Big)\Big(1+\f{\sqrt t}{x}\Big)^{-\nu-1/2}\Big(1+\f{\sqrt t}{y}\Big)^{-\nu-1/2}
		\end{aligned}
		\]
		for $t>0$ and $x\sim y\gtrsim \sqrt t$.
	
	Therefore, we have proved that 
		\begin{equation}\label{eq-1st estimate}
		\f{x}{t}|\delta^{k}_\nu[p_t^\nu(x,y)-p_t^{\nu+1}(x,y)]|\lesi \f{1}{t^{(k+2)/2}}\exp\Big(-\f{|x-y|^2}{ct}\Big)\Big(1+\f{\sqrt t}{x}\Big)^{-\nu-1/2}\Big(1+\f{\sqrt t}{y}\Big)^{-\nu-1/2}
		\end{equation}
				for $t>0$ and $x\sim y\gtrsim \sqrt t$.
				
	We now turn to estimate $\delta^{k+1}_\nu p_t^\nu(x,y)$. To do this,  from \eqref{eq2-delta p}, \eqref{eq-formula for delta k xf} and \eqref{eq- del nu and del nu + 1}, 
	\[
	\begin{aligned}
		|\delta^{k+1}_\nu p_t^\nu(x,y)|&= |\delta^{k}_\nu [\delta_\nu p_t^\nu(x,y)]|\\
		&\lesi \f{1}{t}|\delta_\nu^{k-1}[p_t^\nu(x,y)-p_t^{\nu+1}(x,y)]| +\f{x}{t}|\delta_\nu^{k}[p_t^\nu(x,y)-p_t^{\nu+1}(x,y)]| +\Big|\delta^{k}_\nu\Big[\f{y-x}{2t}p_t^{\nu+1}(x,y)\Big]\Big|. 
	\end{aligned}
	\]
	By using \eqref{eq- inductive hypo delta k pt nu} and \eqref{eq-1st estimate}, we have
	\[
	\begin{aligned}
		\f{1}{t}|\delta_\nu^{k-1}[p_t^\nu(x,y)-p_t^{\nu+1}(x,y)]| &+\f{x}{t}|\delta_\nu^{k}[p_t^\nu(x,y)-p_t^{\nu+1}(x,y)]|\\
		&\lesi \f{1}{t^{(k+2)/2}}\exp\Big(-\f{|x-y|^2}{ct}\Big)\Big(1+\f{\sqrt t}{x}\Big)^{-\nu-1/2}\Big(1+\f{\sqrt t}{y}\Big)^{-\nu-1/2}
	\end{aligned}
	\]
	for $t>0$ and $x\sim y\gtrsim \sqrt t$.
	
	For the last term, applying \eqref{eq-formula for delta k xf}, \eqref{eq- del nu and del nu + 1} and \eqref{eq- inductive hypo delta k pt nu} to obtain
	\[
	\begin{aligned}
		\Big|\delta^{k}_\nu\Big[\f{y-x}{2t}p_t^{\nu+1}(x,y)\Big]\Big|&\lesi \f{1}{t}|\delta^{k-1}_\nu p_t^{\nu+1}(x,y) |+ \f{|y-x|}{t}|\delta^{k}_\nu p_t^{\nu+1}(x,y)|\\
		&\lesi \f{1}{t}|\delta^{k-1}_{\nu+1} p_t^{\nu+1}(x,y) |+\f{1}{tx}|\delta^{k-2}_{\nu+1} p_t^{\nu+1}(x,y) |\\
		& \ \ \ + \f{|y-x|}{t}|\delta^{k}_{\nu+1} p_t^{\nu+1}(x,y)|+ \f{|y-x|}{tx}|\delta^{k-1}_{\nu+1} p_t^{\nu+1}(x,y)|\\
		&\lesi \f{1}{t^{(k+2)/2}}\exp\Big(-\f{|x-y|^2}{ct}\Big)\Big(1+\f{\sqrt t}{x}\Big)^{-\nu-1/2}\Big(1+\f{\sqrt t}{y}\Big)^{-\nu-1/2}
\end{aligned}
\]
for $t>0$ and $x\sim y\gtrsim \sqrt t$.

	This completes our proof.
\end{proof}

From Propositions \ref{prop- delta k p nu first region} and \ref{prop-delta k pt  2nd region}, we have:
\begin{thm}
	\label{thm-delta k pt nu}
	Let $\nu>-1/2$. Then for $\ell \in \mathbb N$ we have
	\[
	|\delta^\ell_\nu p_t^\nu(x,y)|\lesi \f{1}{t^{(\ell+1)/2}}\exp\Big(-\f{|x-y|^2}{ct}\Big)\Big(1+\f{\sqrt t}{x}\Big)^{-\nu-1/2}\Big(1+\f{\sqrt t}{y}\Big)^{-\nu-1/2}
	\]
	for all $t>0$ and $x,y\in (0,\vc)$.
\end{thm}	

\begin{thm}
	\label{thm-partial derivative of delta k pt nu}
	Let $\nu>-1/2$. Then for $\ell, k \in \mathbb N$ we have
	\[
	|\partial_x^k \delta^\ell_\nu p_t^\nu(x,y)|\lesi \Big[\f{1}{t^{k/2}} +\f{1}{x^k} \Big]\f{1}{t^{(\ell+1)/2}}\exp\Big(-\f{|x-y|^2}{ct}\Big)\Big(1+\f{\sqrt t}{x}\Big)^{-\nu-1/2}\Big(1+\f{\sqrt t}{y}\Big)^{-\nu-1/2}
	\]
	and
	\[
	|\partial_y^k \delta^\ell_\nu p_t^\nu(x,y)|\lesi \Big[\f{1}{t^{k/2}} +\f{1}{y^k} \Big]\f{1}{t^{(\ell+1)/2}}\exp\Big(-\f{|x-y|^2}{ct}\Big)\Big(1+\f{\sqrt t}{x}\Big)^{-\nu-1/2}\Big(1+\f{\sqrt t}{y}\Big)^{-\nu-1/2}
	\]
	for all $t>0$ and $x,y\in (0,\vc)$.
\end{thm}	
\begin{proof}
	Since $\delta_\nu(fg)=\delta_\nu f g+f\delta_\nu g$, we have 
	\[
	\partial_x^k=\Big[\delta_\nu +\f{1}{x}\Big(\nu+\f{1}{2}\Big)\Big]^k= \sum_{j=0}^k \f{c_j}{x^j}\delta_\nu^{k-j},
	\]
	where $c_j$ are constants.
	
	This, together with Theorem \ref{thm-delta k pt nu}, implies
	\[
	\begin{aligned}
		|\partial_x^k \delta^\ell_\nu p_t^\nu(x,y)|&\lesi \sum_{j=0}^k \f{1}{x^j}|\delta_\nu^{k+\ell-j}p_t^\nu(x,y)|\\
		&\lesi \Big[\sum_{j=0}^k\f{1}{x^j t^{(k-j)/2}} \Big]\f{1}{t^{(\ell+1)/2}}\exp\Big(-\f{|x-y|^2}{ct}\Big)\Big(1+\f{\sqrt t}{x}\Big)^{-\nu-1/2}\Big(1+\f{\sqrt t}{y}\Big)^{-\nu-1/2}.
	\end{aligned}
	\]
	Using the following inequality
	\[
	\sum_{j=0}^k\f{1}{x^j t^{(k-j)/2}}\lesi \f{1}{t^{k/2}} +\f{1}{x^k},
	\]
	we further imply
	\[
	\begin{aligned}
		|\partial_x^k \delta^\ell_\nu p_t^\nu(x,y)| \lesi \Big[\f{1}{t^{k/2}} +\f{1}{x^k} \Big]\f{1}{t^{(\ell+1)/2}}\exp\Big(-\f{|x-y|^2}{ct}\Big)\Big(1+\f{\sqrt t}{x}\Big)^{-\nu-1/2}\Big(1+\f{\sqrt t}{y}\Big)^{-\nu-1/2}
	\end{aligned}
	\]
		for all $t>0$ and $x,y\in (0,\vc)$.
	Similarly, we have
	\[
	|\partial_y^k \delta^\ell_\nu p_t^\nu(x,y)|\lesi \Big[\f{1}{t^{k/2}} +\f{1}{y^k} \Big]\f{1}{t^{(\ell+1)/2}}\exp\Big(-\f{|x-y|^2}{ct}\Big)\Big(1+\f{\sqrt t}{x}\Big)^{-\nu-1/2}\Big(1+\f{\sqrt t}{y}\Big)^{-\nu-1/2}
	\]
		for all $t>0$ and $x,y\in (0,\vc)$.

		This completes our proof.
\end{proof}

	
\begin{cor}\label{cor1}
	Let $\nu>-1/2$. Then for each $k, M\in \mathbb N$ we have
	\begin{equation}\label{eq1-cor1}
	|\delta^k_\nu \Delta^M_\nu p_t^\nu(x,y)|\lesi \f{1}{t^{(k+2M+1)/2}}\exp\Big(-\f{|x-y|^2}{ct}\Big)
	\end{equation}
	and
	\begin{equation}\label{eq2-cor1}
	|\Delta_\nu^M (\delta^*_\nu)^kp_t^{\nu+k+2M}(x,y)|\lesi \f{1}{t^{(k+2M+1)/2}}\exp\Big(-\f{|x-y|^2}{ct}\Big)
	\end{equation}
	for all $x,y\in \mathbb R_+$ and $t>0$.
\end{cor}
\begin{proof}
	We first write 
	\[
	\delta^k_\nu \Delta^M_\nu e^{-t\Delta_\nu}= \delta^k_\nu e^{-\f{t}{2}\Delta_\nu}\circ \Delta^M_\nu e^{-\f{t}{2}\Delta_\nu},
	\]
	which implies
	\[
	\delta^k_\nu \Delta^M_\nu p_t^\nu(x,y)=(-1)^M\int_{\mathbb R_+}\delta^k_\nu p_{t/2}^\nu(x,z) \partial_t^Mp_{t/2}^\nu(z,y)dz.
	\]
	By Theorem \ref{thm-delta k pt nu},
	\[
	|\delta^k_\nu p_{t/2}^\nu(x,z)|\lesi \f{1}{t^{(k+1)/2}}\exp\Big(-\f{|x-z|^2}{ct}\Big)
	\]
	for all $x,z\in \mathbb R_+$ and $t>0$.
	
	On the other hand, from the Gaussian upper bound of $p_t^\nu(x,y)$ in Theorem \ref{thm-heat kernel} and \cite[Lemma 2.5]{CD}, we have
	\[
	|\partial_t^Mp_{t/2}^\nu(z,y)|\lesi \f{1}{t^{(2M+1)/2}}\exp\Big(-\f{|z-y|^2}{ct}\Big)
	\]
	for all $z,y\in \mathbb R_+$ and $t>0$.
	
	Therefore,
	\[
	\begin{aligned}
		|\delta^k_\nu \Delta^M_\nu p_t^\nu(x,y)|&\lesi \int_{\mathbb R_+} \f{1}{t^{(k+1)/2}}\exp\Big(-\f{|x-z|^2}{ct}\Big)\f{1}{t^{(2M+1)/2}}\exp\Big(-\f{|z-y|^2}{ct}\Big)dz\\
		&\lesi \f{1}{t^{(k+2M+1)/2}}\exp\Big(-\f{|x-y|^2}{ct}\Big)
	\end{aligned}
	\]
	for all $x,y\in \mathbb R_+$ and $t>0$, which ensures \eqref{eq1-cor1}.
	
	For \eqref{eq2-cor1}, we first have
	\[
	\delta_\nu = \delta_{\nu+k+2M} + \f{k+2M}{x}
	\]
	and
	\[
	\delta_\nu^* = -\delta_\nu -\f{2\nu+1}{x}=-\delta_{\nu+k+2M}-\f{2\nu+k+2M+1}{x}.
	\]
	Hence,
	\[
	\begin{aligned}
		\Delta_\nu^M (\delta^*_\nu)^k &= (\delta_\nu^*\delta_\nu)^M(\delta^*_\nu)^k\\
		&= \Big[\Big(-\delta_{\nu+k+2M}-\f{2\nu+k+2M+1}{x}\Big) \Big( \delta_{\nu+k+2M} + \f{k+2M}{x}\Big)\Big]^M\\
		& \ \ \Big( -\delta_{\nu+k+2M}-\f{2\nu+k+2M+1}{x}\Big)^k.
	\end{aligned}
	\]
	Using the fact $\delta_\nu(fg)=\delta_\nu f g + f' g$, we further implies
	\[
	\Delta_\nu^M (\delta^*_\nu)^k=\sum_{j=0}^{2M+k} \f{c_j}{x^j}\delta^{2M+k-j}_{\nu+k+2M},
	\]
	where $c_j$ are constants.
	
	It follows that
	\[
	\begin{aligned}
		|\Delta_\nu^M (\delta^*_\nu)^kp_t^{\nu+k+2M}(x,y)|&\lesi \sum_{j=0}^{2M+k} \f{1}{x^j}|\delta^{2M+k-j}_{\nu+k+2M} p_t^{\nu+k+M}(x,y)|,
	\end{aligned}
	\] 
	which, together with Theorem \ref{thm-delta k pt nu}, implies
	\[
	\begin{aligned}
		|\Delta_\nu^M (\delta^*_\nu)^kp_t^{\nu+k+2M}(x,y)|&\lesi \sum_{j=0}^{2M+k} \f{1}{x^j}\f{1}{t^{(2M+k-j+1)/2}}\exp\Big(-\f{|x-y|^2}{ct}\Big)\Big(1+\f{\sqrt t}{x}\Big)^{-j}\\
		&   \lesi \f{1}{t^{(2M+k+1)/2}}\exp\Big(-\f{|x-y|^2}{ct}\Big)
	\end{aligned}
	\] 
	for all $x,y \in \mathbb R_+$ and $t>0$.
	
	This completes our proof.
\end{proof}

The following results play an important role in proving the $L^2$-boundedness of the higher-order Riesz transforms.
\begin{prop}\label{prop- difference two riesz kernels}
	Let $\nu\in (-1/2,\vc)$. Then for $\ell\in \mathbb N$ we have
	\[
	|\delta_\nu^{\ell} p_t^\nu(x,y)-\delta_{\nu+1}^{\ell} p_t^{\nu+1}(x,y)|\lesi \f{1}{xt^{\ell/2}}\exp\Big(-\f{|x-y|^2}{ct}\Big)
	\]
	for all $y/2<x<2y$ and $x\ge \sqrt t$.
\end{prop}
\begin{proof}
	
	We will prove the inequality  by induction.
	
	For $\ell=0$, by \eqref{eq5-Inu} and Theorem \ref{thm-heat kernel} we have
	\[
	\begin{aligned}
		|\delta_\nu^{\ell} p_t^\nu(x,y)-\delta_{\nu+1}^{\ell} p_t^{\nu+1}(x,y)|&= |p_t^\nu(x,y)- p_t^{\nu+1}(x,y)|\\
		&\lesi \f{t}{xy}\f{1}{\sqrt t}\exp\Big(-\f{|x-y|^2}{ct}\Big)\\
		&\lesi \f{1}{x}\exp\Big(-\f{|x-y|^2}{ct}\Big),
	\end{aligned}
	\]
	as long as $y/2<x<2y$ and $x\ge \sqrt t$.
	
	This ensures the proposition for the case $\ell=0$.
	
	Assume the proposition is true for $\ell=0,1,\ldots, k$. That is, for $\ell=0,1,\ldots, k$, we have
	\[
	|\delta_\nu^{\ell} p_t^\nu(x,y)-\delta_{\nu+1}^{\ell} p_t^{\nu+1}(x,y)|\lesi \f{1}{xt^{\ell/2}}\exp\Big(-\f{|x-y|^2}{ct}\Big)
	\]
	for all $y/2<x<2y$ and $x\ge \sqrt t$.
	
	We need to prove the estimate for $\ell=k+1$. Using \eqref{eq- del nu and del nu + 1}, we have
	\[
	\begin{aligned}
		|\delta^{k+1}_\nu p_t^\nu(x,y) - \delta^{k+1}_{\nu+1} p_t^{\nu+1}(x,y)|&= \Big|\delta^{k}_\nu[\delta_\nu p_t^\nu(x,y) - \delta_{\nu+1}p_t^{\nu+1}(x,y)] +\f{1}{x}\delta_{\nu+1}^{k}p_t^{\nu+1}(x,y)\Big| \\
		&\lesi \Big|\delta^{k}_\nu[\delta_\nu p_t^\nu(x,y) - \delta_{\nu+1}p_t^{\nu+1}(x,y)]\Big| +\f{1}{x}|\delta_{\nu+1}^{k}p_t^{\nu+1}(x,y)|\\
		&=: E_1 +E_2.
	\end{aligned}
	\]
	By Theorem \ref{thm-delta k pt nu},
	\[
	E_2\lesi \f{1}{x}\f{1}{t^{(k+1)/2}}\exp\Big(-\f{|x-y|^2}{ct}\Big).
	\]	
	It remains to estimate $E_1$. From \eqref{eq2-delta p} and \eqref{eq6-Inu},
	\begin{equation}\label{eq-del pt new}
		\begin{aligned}
			\delta_\nu p_t^\nu(x,y)-\delta_{\nu+1} p_t^{\nu+1}(x,y)&=-\f{x}{2t}[p_t^\nu(x,y)-p_t^{\nu+2}(x,y)]+\f{y-x}{2t}[p_t^{\nu+1}(x,y)-p_t^{\nu+2}(x,y)]\\
			&= \f{2(\alpha+1)}{y} p_t^{\nu+1}(x,y)+\f{y-x}{2t}[p_t^{\nu+1}(x,y)-p_t^{\nu+2}(x,y)].
		\end{aligned}
	\end{equation}
	
	It, together with \eqref{eq-formula for delta k xf}, follows that 
	\[
	\begin{aligned}
		E_1 &\lesi  \f{1}{y} \delta_\nu^k p_t^{\nu+1}(x,y) +\f{1}{t}\big|\delta_\nu^{k-1}[p_t^{\nu+1}(x,y)-p_t^{\nu+2}(x,y)]\big| +\f{|y-x|}{t} \big|\delta_\nu^{k}[p_t^{\nu+1}(x,y)-p_t^{\nu+2}(x,y)]\big|\\
		&=:E_{11}+E_{12}+E_{13}.
	\end{aligned}
	\]	
	Using \eqref{eq- del nu and del nu + 1} and Theorem \ref{thm-delta k pt nu},
	\[
	\begin{aligned}
		E_{11}&\lesi \f{1}{y}|\delta_{\nu+1}^k p_t^{\nu+1}(x,y)|+\f{1}{xy}|\delta_{\nu+1}^{k-1} p_t^{\nu+1}(x,y)|\\
		&\lesi \f{1}{x} \f{1}{t^{(k+1)/2}}\exp\Big(-\f{|x-y|^2}{ct}\Big),
	\end{aligned}
	\]		
	as long as $y/2<x<2y$ and $x\ge \sqrt t$.
	
	For $E_{12}$, by \eqref{eq- del nu and del nu + 1} and Proposition \ref{prop-delta k pt  2nd region}, we have
	\[
	\begin{aligned}
		E_{12} &\lesi  \f{1}{t}\big|\delta_{\nu+1}^{k-1}[p_t^{\nu+1}(x,y)-p_t^{\nu+2}(x,y)]\big|+\f{1}{tx}\big|\delta_{\nu+1}^{k-2}[p_t^{\nu+1}(x,y)-p_t^{\nu+2}(x,y)]\big|\\
		&\lesi  \f{1}{x} \f{1}{t^{(k+1)/2}}\exp\Big(-\f{|x-y|^2}{ct}\Big),
	\end{aligned}
	\]
	as long as $y/2<x<2y$ and $x\ge \sqrt t$.

	Similarly,
	\[
	\begin{aligned}
		E_{13}&\lesi \f{|x-y|}{t}\f{1}{x t^{k/2}}\exp\Big(-\f{|x-y|^2}{ct}\Big)\\
		&\lesi  \f{1}{x} \f{1}{t^{(k+1)/2}}\exp\Big(-\f{|x-y|^2}{2ct}\Big),
	\end{aligned}
	\]
	as long as $x\sim y\gtrsim \sqrt t$.
	
	This completes our proof.
	
\end{proof}

\begin{prop}\label{prop- difference Riesz}
	Let $\nu\in (-1/2,\vc)$ and $k\in \mathbb N$. Then for any $\epsilon>0$, we have
	\begin{equation}\label{eq1- difference Riesz kernel}
	\int_0^\vc t^{k/2}|\delta_\nu^kp_t^{\nu}(x,y)- \delta_{\nu+1}^kp_t^{\nu+1}(x,y)|\f{dt}{t}\lesi \Big[\f{1}{x} + \f{1}{x}\Big(\f{x}{|x-y|}\Big)^{\epsilon}\Big]\chi_{\{y/2<x<2y\}} + \f{1}{x}\chi_{\{x\ge 2y\}} + \f{1}{y}\chi_{\{y\ge 2x\}}.
	\end{equation}
	Consequently, the  operator
	\[
	f\mapsto  \int_0^\vc t^{k/2}\big|[\delta_\nu^k e^{-t\Delta_\nu}- \delta_{\nu+1}^ke^{-t\Delta_{\nu+1}}]f\big|\f{dt}{t} 
	\]
	is bounded on $L^p(\mathbb R_+)$ for all $1<p<\vc$.
\end{prop}
\begin{proof}
	
	If $y/2<x<2x$, then we have
	\[
	\int_0^\vc t^{k/2}|\delta_\nu^kp_t^{\nu+1}(x,y)- \delta_{\nu+1}^kp_t^\nu(x,y)|\f{dt}{t}=\int_0^{x^2} \ldots + \int_{x^2}^\vc \ldots
	\]
	For the first term, using Theorem \ref{thm-delta k pt nu},
	\[
	\begin{aligned}
		\int_{x^2}^\vc t^{k/2}|\delta_\nu^kp_t^{\nu+1}(x,y)- \delta_{\nu+1}^kp_t^\nu(x,y)|\f{dt}{t}&\le \int_{x^2}^\vc t^{k/2}|\delta_\nu^kp_t^{\nu+1}(x,y)|\f{dt}{t}+\int_{x^2}^\vc t^{k/2}| \delta_{\nu+1}^kp_t^\nu(x,y)|\f{dt}{t}\\
		&\lesi \int_{x^2}^\vc \f{1}{\sqrt t}\Big(\f{\sqrt t}{x}\Big)^{-1-2\nu}\f{dt}{t}\\
		&\sim \int_{x^2}^\vc \f{x^{1+2\nu}}{t^{1+\nu}}\f{dt}{t}\sim \f{1}{x}.
	\end{aligned}
	\]
	For the second part, using Proposition \ref{prop- difference two riesz kernels},
	\[
	\begin{aligned}
		\int^{x^2}_0 t^{k/2}|\delta_\nu^kp_t^{\nu+1}(x,y)- \delta_{\nu+1}^kp_t^\nu(x,y)|\f{dt}{t}&\lesi \int^{x^2}_0\f{1}{x}\exp\Big(-\f{|x-y|^2}{ct}\Big)\f{dt}{t} \\
		&\lesi \int^{x^2}_0\f{1}{x}\Big(\f{\sqrt t}{|x-y|}\Big)^{\epsilon}\f{dt}{t}\\
		&\lesi  \f{1}{x}\Big(\f{x}{|x-y|}\Big)^{\epsilon}.
	\end{aligned}
	\]
	
	\medskip
	
	If $x\ge 2y$, then $|x-y|\sim x$. This, together with Theorem \ref{thm-delta k pt nu}, implies
	\[
	\begin{aligned}
		\int_0^\vc t^{k/2}|\delta_\nu^kp_t^{\nu+1}(x,y)&- \delta_{\nu+1}^kp_t^\nu(x,y)|\f{dt}{t}\\
		&\le 
		\int_0^\vc t^{k/2}|\delta_\nu^kp_t^{\nu+1}(x,y)|\f{dt}{t}+
		\int_0^\vc t^{k/2}| \delta_{\nu+1}^kp_t^\nu(x,y)|\f{dt}{t}\\
		&\lesi \int_0^\vc \f{1}{\sqrt t}\exp\Big(-\f{x^2}{ct}\Big)\Big[\Big(1+\f{\sqrt t}{x}\Big)^{-\nu-1/2}+\Big(1+\f{\sqrt t}{x}\Big)^{-\nu-3/2}\Big]\f{dt}{t}\\
		&\lesi \int_0^\vc \f{1}{\sqrt t}\exp\Big(-\f{x^2}{ct}\Big) \Big(1+\f{\sqrt t}{x}\Big)^{-\nu-1/2} \f{dt}{t}\\
		&\lesi \f{1}{x}.
	\end{aligned}
	\]
	Similarly, if $y\ge 2x$, then we have
	\[
	\int_0^\vc t^{k/2}|\delta_\nu^kp_t^{\nu+1}(x,y)- \delta_{\nu+1}^kp_t^\nu(x,y)|\f{dt}{t}\lesi \f{1}{y}.
	\]
	This completes the proof of \eqref{eq1- difference Riesz kernel}.
	
	For the second part, from the inequality \eqref{eq1- difference Riesz kernel} we have,
	\[
	\begin{aligned}
		\int_0^\vc t^{k/2}&\big|[\delta_\nu^k e^{-t\Delta_\nu}- \delta_{\nu+1}^ke^{-t\Delta_{\nu+1}}]f(x)\big|\f{dt}{t} \\
		&=\int_0^\vc\int_0^\vc t^{k/2}\big|[\delta_\nu^kp_t^{\nu}(x,y)- \delta_{\nu+1}^kp_t^{\nu+1}(x,y)]f(y)\big|\f{dt}{t}dy\\
		&\lesi \int_{x/2}^{2x}\Big[\f{1}{x} + \f{1}{x}\Big(\f{x}{|x-y|}\Big)^{\epsilon}\Big]|f(y)|dy + \int_0^{x/2}\f{1}{x}|f(y)|dy + \int_{2x}^\vc \f{1}{y}|f(y)|dy\\
		&=: T_1f(x) + T_2f(x) +T_3f(x). 
	\end{aligned}
	\]
	Obviously,
	\[
	T_2f(x)\le 	2\int_0^{2x}\f{1}{2x}|f(y)|dy \le 2\mathcal Mf(x),
	\]
	which, together with the $L^p$-boundedness of $\mathcal M$, implies that $T_2$ is bounded on $L^p(\mathbb R_+)$.
	
	For $T_3$, let $g\in L^{p'}(\mathbb R_+)$. Then we have
	\[
	\begin{aligned}
		\langle T_3f, g\rangle &=\int_0^\vc \int_{2x}^\vc \f{1}{y}|f(y)|g(x)dydx\\
		&\lesi \int_0^\vc |f(y)|\int_{0}^{y/2} \f{1}{y}|g(x)|dxdy.
	\end{aligned}
	\]
	Similarly,
	\[
	\int_{0}^{y/2} \f{1}{y}|g(x)|dx\lesi \mathcal Mg(y).
	\]
	Hence,
	\[
	\begin{aligned}
		\langle T_3f, g\rangle  
		&\lesi \int_0^\vc |f(y)| \mathcal Mg(y) dy\\
		&\lesi \|f\|_p \|\mathcal Mg\|_{p'}\\
		&\lesi \|f\|_p \|g\|_{p'}.
	\end{aligned}
	\]
	It follows that $T_3$ is bounded on $L^p(\mathbb R_+)$.
	
	It remains to show that $T_1$ is bounded on $L^p(\mathbb R_+)$. To do this, fix    $1<r<p$ and $\epsilon<1/r'$. Then we have
	\[
	\begin{aligned}
		T_1 f(x) &\lesi \f{1}{x}\int_{x/2}^{2x}|f(y)|dy + \f{1}{x}\int_{x/2}^{2x}\Big(\f{x}{|x-y|}\Big)^\epsilon |f(y)|dy.
	\end{aligned}
	\]
	Obviously,
	\[
	\f{1}{x}\int_{x/2}^{2x}|f(y)|dy\lesi \mathcal Mf(x),
	\]
	and hence 
	$$
	f\mapsto \f{1}{x}\int_{x/2}^{2x}|f(y)|dy
	$$
	is bounded on $L^p(\mathbb R_+)$.
	
	For the second term, by H\"older's inequality and the fact $r'\epsilon<1$,
	\[
	\begin{aligned}
		\f{1}{x}\int_{x/2}^{2x}\Big(\f{x}{|x-y|}\Big)^\epsilon |f(y)|dy&\lesi \Big(\f{1}{x}\int_{x/2}^{2x} |f(y)|^{r}dy\Big)^{1/r}\Big(\f{1}{x}\int_{x/2}^{2x} \Big(\f{x}{|x-y|}\Big)^{r'\epsilon}dy\Big)^{1/r'}\\
		&\lesi \Big(\f{1}{x}\int_{x/2}^{2x} |f(y)|^{r}dy\Big)^{1/r}\\
		&\lesi \mathcal M_rf,
	\end{aligned}
	\]
	which implies the operator
	\[
	f\mapsto \f{1}{x}\int_{x/2}^{2x}\Big(\f{x}{|x-y|}\Big)^\epsilon |f(y)|dy
	\]
	is bounded on $L^p(\mathbb R_+)$.
	
	This completes our proof.
\end{proof}

	\subsection{The case $n\ge 2$} For  $\nu=(\nu_1,\ldots,\nu_n)\in (-1/2,\vc)^n$, recall that
	\[
	\nu_{\min} =  \min\{\nu_j: j =1,\ldots, n\}.
	\]

From Theorems	\ref{thm-delta k pt nu} and \ref{thm-partial derivative of delta k pt nu}  we have the following two propositions.	
	\begin{prop}
		\label{prop- delta k pt d>2} Let $\nu\in (-1/2,\vc)^n$ and $\ell\in \mathbb N^n$. Then we have
		\begin{equation*}
			| \delta_\nu^\ell  p_t^\nu(x,y)|\lesi   \f{1}{t^{(n+|\ell|)/2}}\exp\Big(-\f{|x-y|^2}{ct}\Big)\Big(1+\f{\sqrt t}{\rho(x)}+\f{\sqrt t}{\rho(y)}\Big)^{-(\nu_{\min}+1/2)}
		\end{equation*}
		for $t>0$ and $x,y\in \mathbb R^n_+$.
	\end{prop}

	\begin{prop}\label{prop-gradient x y d>2}
		Let $\nu\in (-1/2,\vc)^n$ and $k, \ell\in \mathbb N^n$.  Then, we have
		\begin{equation*}
			\begin{aligned}
				|\partial_{x}^k \delta_\nu^\ell p_t^\nu(x,y)|&\lesi \Big[\f{1}{t^{|k|/2}}\ + \f{1}{\rho(x)^{|k|}}\Big]\f{1 }{ t^{(n+|\ell|)/2}}\exp\Big(-\f{|x-y|^2}{ct}\Big)\Big(1+\f{\sqrt t}{\rho(x)}+\f{\sqrt t}{\rho(y)}\Big)^{-(\nu_{\min}+1/2)}
			\end{aligned}
		\end{equation*}
		and
		\begin{equation*}
			\begin{aligned}
				|\partial_{y}^k \delta_\nu^\ell p_t^\nu(x,y)|&\lesi \Big[\f{1}{t^{|k|/2}}\ + \f{1}{\rho(y)^{|k|}}\Big]\f{1 }{ t^{(n+\ell)/2}}\exp\Big(-\f{|x-y|^2}{ct}\Big)\Big(1+\f{\sqrt t}{\rho(x)}+\f{\sqrt t}{\rho(y)}\Big)^{-(\nu_{\min}+1/2)}
			\end{aligned}
		\end{equation*}
		for all $t>0$ and all $x, y\in \mathbb R^n_+$.
	\end{prop}
	
	From Corollary \ref{cor1}, we have:
	\begin{cor}\label{cor1  d ge 2}
		Let $\nu\in (-1/2,\vc)^n$. Then for each $k, \vec{M}=(M,\ldots, M)\in \mathbb N^n$ we have
		\begin{equation}\label{eq1-cor2}
			|\delta^k_\nu \Delta^M_\nu p_t^\nu(x,y)|\lesi \f{1}{t^{(|k|+2M+n)/2}}\exp\Big(-\f{|x-y|^2}{ct}\Big)
		\end{equation}
		and
		\begin{equation}\label{eq2-cor2}
			|\Delta_\nu^M (\delta^*_\nu)^kp_t^{\nu+k+2\vec M}(x,y)|\lesi \f{1}{t^{(|k|+2M+n)/2}}\exp\Big(-\f{|x-y|^2}{ct}\Big)
		\end{equation}
		for all $x,y\in \mathbb R^n_+$ and $t>0$.
	\end{cor}
	
\section{Hardy spaces associated to the Laguerre operator and its duality}

This section is dedicated to proving Theorem \ref{mainthm2s} and Theorem \ref{mainthm-dual}.

From the definition of the critical function $\rho$ in \eqref{eq- critical function}, it is easy to see that if $y\in B(x,4\rho(x))$, then $\rho(x)\sim \rho(y)$. We will use this frequently without any explanation.

\bigskip

We now give the proof of Theorem \ref{mainthm2s}. 
\begin{proof}[Proof of Theorem \ref{mainthm2s}:]
	We divide the proof into two steps: $H^{p}_{\rho}(\mathbb{R}^n_+)\hookrightarrow H^p_{\Delta_\nu}(\mathbb{R}^n_+)$ and $H^p_{\Delta_\nu}(\mathbb{R}^n_+)\hookrightarrow H^{p}_{\rho}(\mathbb{R}^n_+)$.

	\bigskip
	
	\noindent \textbf{Step 1: Proof of $H^{p}_{{\rm at},\rho}(\mathbb{R}^n_+)\hookrightarrow H^p_{\Delta_\nu}(\mathbb{R}^n_+)$.} Fix $\f{n}{n+\gamma_\nu}<p\le 1$. Let $a$ be a $(p,\rho)$ atom associated with a ball $B:=B(x_0,r)$. By Remark \ref{rem1}, we might assume that $r\le \rho(x_0)$. By the definition of $H^p_{\Delta_\nu}(\mathbb{R}^n_+)$, it suffices to prove that  
	\[
	\|\mathcal M_{\Delta_\nu} a\|_p \lesi 1.
	\]
	To do this, we write
	\[
	\begin{aligned}
		\|\mathcal M_{\Delta_\nu} a\|_p&\lesi \|\mathcal M_{\Delta_\nu} a\|_{L^p(4B)} +\|\mathcal M_{\Delta_\nu} a\|_{L^p((4B)^c)}\\
		&\lesi E_1 + E_2.
	\end{aligned}
	\]
	Since the kernel of $e^{-t\Delta_\nu}$ satisfies a Gaussian upper bound (see Theorem \ref{thm-heat kernel}), the maximal function $\mathcal M_{\Delta_\nu}$ is bounded on $L^q(\Rn_+), 1<q<\vc$. This, along with the H\"older  inequality, implies
	\[
	\begin{aligned}
		\|\mathcal M_{\Delta_\nu} a\|_{L^p(4B)}&\lesi |4B|^{1/p-1/2} \|\mathcal M_{\Delta_\nu} a\|_{L^2(4B)}\\
		&\lesi |4B|^{1/p-1/2} \|a\|_{L^2(B)}\\
		&\lesi 1.
	\end{aligned}
	\]
	It remains to take care of the second term $E_2$. We now consider two cases.
	\medskip
	
	\textbf{Case 1: $r=\rho(x_0)$.} By Theorem  \ref{thm-heat kernel}, for $x\in (4B)^c$,
	\[
	\begin{aligned}
		|\mathcal M_{\Delta_\nu} a(x)|\lesi \sup_{t>0}\int_{B}\f{1}{t^{n/2}}\exp\Big(-\f{|x-y|^2}{ct}\Big)\Big(\f{\rho(y)}{\sqrt t}\Big)^{\gamma_\nu}|a(y)|dy,
	\end{aligned}
	\]
	where $\gamma_\nu=\nu_{\min}+1/2$.

	Since $\rho(y)\sim \rho(x_0)$ for $y\in B$ and $|x-y|\sim |x-x_0|$ for $x\in (4B)^c$ and $y\in B$, we further imply
	\[
	\begin{aligned}
		\mathcal M_{\Delta_\nu} a(x)&\lesi \sup_{t>0}\int_{B}\f{1}{t^{n/2}}\exp\Big(-\f{|x-x_0|^2}{ct}\Big)\Big(\f{\rho(x_0)}{\sqrt t}\Big)^{\gamma_\nu}|a(y)|dy\\
		&\lesi \Big(\f{\rho(x_0)}{|x-x_0|}\Big)^{\gamma_\nu}\f{1}{|x-x_0|^n}\|a\|_1\\
		&\lesi \Big(\f{r}{|x-x_0|}\Big)^{\gamma_\nu}\f{1}{|x-x_0|^n}|B|^{1-1/p}.
	\end{aligned}
	\]
	Therefore,
	\[
	\begin{aligned}
		\|\mathcal M_{\Delta_\nu} a\|_{L^p((4B)^c)}&\lesi |B|^{1-1/p} \Big[\int_{(4B)^c}\Big(\f{r}{|x-x_0|}\Big)^{p\gamma_\nu}\f{1}{|x-x_0|^{np}}dx\Big]^{1/p}\\
		&\lesi 1,
	\end{aligned}
	\]
	as long as $\f{n}{n+\gamma_\nu}<p\le 1$.
	
	\bigskip
	

\textbf{Case 2: $r<\rho(x_0)$.} Using the cancellation property $\int a(x) x^\alpha dx= 0$ for all $|\alpha|\le \lfloor n(1/p-1)\rfloor=:N_p$ and the Taylor expansion, we have
\[
\begin{aligned}
	\sup_{t>0} |e^{-t\Delta_\nu}a(x)|&= \sup_{t>0}\Big|\int_{B}[p_t^\nu(x,y)-p_t^\nu(x,x_0)]a(y)dy\Big|\\
	&= \sup_{t>0}\Big|\int_{B}\Big[p_t^\nu(x,y)-\sum_{|\alpha|\le N_p}\f{\partial_y^\alpha p_t^\nu(x,x_0)}{\alpha!}(y-x_0)^\alpha\Big]a(y)dy\Big|\\
	&= \sup_{t>0}\Big|\int_{B} \sum_{|\alpha|= N_p+1}\f{\partial_y^\alpha p_t^\nu(x,x_0+\theta(y-x_0))}{\alpha!}(y-x_0)^\alpha a(y)dy\Big|
\end{aligned}
\]
for some $\theta \in (0,1)$.

Since $y\in B$ we have $\rho(x_0+\theta(y-x_0))\sim \rho(x_0)$ and $|x-[x_0+\theta(y-x_0)]|\sim |x-x_0|$ for all $x\in (4B)^c$, $y\in B$ and $\theta \in (0,1)$, by  Proposition \ref{prop-gradient x y d>2} we further imply, for $x\in (4B)^c$,
\[
\begin{aligned}
	\sup_{t>0} |e^{-t\Delta_\nu}a(x)|&\lesi  \sup_{t>0}\int_{B}\Big(\f{|y-y_0|}{\sqrt t}+\f{|y-y_0|}{\rho(x_0)}\Big)^{N_p+1}\f{1}{t^{n/2}}\exp\Big(-\f{|x-y|^2}{ct}\Big)\Big(1+\f{\sqrt t}{\rho(x_0)}\Big)^{-\gamma_\nu} |a(y)|dy\\
	&\sim  \sup_{t>0}\int_{B}\Big(\f{r}{\sqrt t}+\f{r}{\rho(x_0)}\Big)^{N_p+1}\f{1}{t^{n/2}}\exp\Big(-\f{|x-x_0|^2}{ct}\Big)\Big(1+\f{\sqrt t}{\rho(x_0)}\Big)^{-\gamma_\nu} \|a\|_1\\
	&\lesi   \sup_{t>0}  \Big(\f{r}{\sqrt t}\Big)^{N_p+1} \f{1}{t^{n/2}}\exp\Big(-\f{|x-x_0|^2}{ct}\Big) \|a\|_1\\
	& \ \ +   \sup_{t>0}  \Big(\f{r}{\rho(x_0)}\Big)^{N_p+1} \f{1}{t^{n/2}}\exp\Big(-\f{|x-x_0|^2}{ct}\Big)\Big(1+\f{\sqrt t}{\rho(x_0)}\Big)^{-\gamma_\nu} \|a\|_1\\	
	&= F_1 + F_2.
\end{aligned}
\]
For the term $F_1$, it is straightforward to see that 
\[
\begin{aligned}
	F_1&\lesi \Big(\f{r}{|x-x_0|}\Big)^{N_p+1}\f{1}{|x-x_0|^n}|B|^{1-1/p}\\
	&\lesi \Big(\f{r}{|x-x_0|}\Big)^{(N_p+1)\wedge \gamma_\nu}\f{1}{|x-x_0|^n}|B|^{1-1/p},	
\end{aligned}
\]
where in the last inequality we used the fact $r\le |x-x_0|$ for $x\in (4B)^c$.

For $F_2$, since $r<\rho(x_0)$, we have
\[
\begin{aligned}
	F_2&\lesi \sup_{t>0} \Big(\f{r}{\rho(x_0)}\Big)^{(N_p+1)\wedge \gamma_\nu}\f{1}{t^{n/2}}\exp\Big(-\f{|x-x_0|^2}{ct}\Big)\Big(\f{\rho(x_0)}{\sqrt t}\Big)^{(N_p+1)\wedge \gamma_\nu} \|a\|_1\\
	&\lesi \Big(\f{r}{|x-x_0|}\Big)^{^{(N_p+1)\wedge \gamma_\nu}}\f{1}{|x-x_0|^n}|B|^{1-1/p}.
\end{aligned}
\]
Taking this and the estimate of $F_1$ into account then we obtain
\[
\sup_{t>0} |e^{-t\Delta_\nu}a(x)|\lesi \Big(\f{r}{|x-x_0|}\Big)^{^{(N_p+1)\wedge \gamma_\nu}}\f{1}{|x-x_0|^n}|B|^{1-1/p}.
\]
Therefore,
\[
|\mathcal M_{\Delta_\nu} a(x)| \lesi \Big(\f{r}{|x-x_0|}\Big)^{^{(N_p+1)\wedge \gamma_\nu}}\f{1}{|x-x_0|}|B|^{1-1/p},
\]
which implies
\[
\|\mathcal M_{\Delta_\nu} a\|_p\lesi 1,
\]
as long as $\f{n}{n+\gamma_\nu}< p\le 1$.

This completes the proof of the first step.

\bigskip

\noindent \textbf{Step 2: Proof of $H^p_{\Delta_\nu}(\mathbb{R}^n_+)\hookrightarrow  H^{p}_{\rho}(\mathbb{R}^n_+)$.} 

Recall from \cite{SY} that for $p\in (0,1]$ and $N\in \mathbb N$, a function $a$ is call a $(p,N)_{\Delta_\nu}$ atom associated to a ball $B$  if  
\begin{enumerate}[{\rm (i)}]
	\item  $a=\Delta_\nu^N b$;
	\item $\supp \Delta_\nu^{k}b\subset B, \ k=0, 1, \dots, M$;
	\item $\|\Delta_\nu^{k}b\|_{L^\vc(\mathbb{R}^n_+)}\leq
	r_B^{2(N-k)}|B|^{-\f{1}{p}},\ k=0,1,\dots,N$.
\end{enumerate}

Let $f\in H^p_{\Delta_\nu}(\mathbb{R}^n_+)\cap L^2(\Rn_+)$. Since $\Delta_\nu$ is a nonnegative self-adjoint operator and satisfies the Gaussian upper bound, by Theorem 1.3 in \cite{SY}, we can write $f = \sum_{j}\lambda_ja_j$ in $L^2(\Rn_+)$, where 
$\sum_{j}|\lambda_j|^p\sim \|f\|^p_{H^p_{\Delta_\nu}(\mathbb{R}^n_+)}$ and each $a_j$ is a $(p,N)_{\Delta_\nu}$ atom with $N>n(\f{1}{p}-1)$.  Therefore, it suffices to prove that $a\in H^p_{\rho}(\Rn_+)$ for any $(p,N)_{\Delta_\nu}$ atom associated to a ball $B$ with $N>n(\f{1}{p}-1)$. 

If $r_B\ge \rho(x_B)$, then from (iii), $a$ is also a $(p,\rho)$ atom and hence $a\in H^p_{\rho}(\Rn_+)$. Hence, it remains to consider the case  $r_B<\rho(x_B)$. 

We first claim that for any multi-index $\alpha$ with $|\alpha|<N$, we have 
\begin{equation}\label{eq-integral of a}
	\Big|\int (x-x_B)^\alpha a(x)\,dx\Big| \le |B|^{1-\f{1}{p}} r_B^{|\alpha|} \Big(\f{r_B}{\rho_B}\Big)^{N-|\alpha|}.
\end{equation} 
Indeed, from (i) we have
\[
\begin{aligned}
	\Big|\int (x-x_B)^\alpha a(x)\,dx\Big|&=\Big|\int_B (x-x_B)^\alpha \Delta_\nu^N b (x)\,dx\Big|\\
	&=\Big|\int_B \Delta_\nu^N(x-x_B)^\alpha  b (x)\,dx\Big|\\
	&\le  \int_B |\Delta_\nu^N(x-x_B)^\alpha| |b (x)|\,dx.
\end{aligned}
\]
Note that 
\[
\Delta_\nu^N(x-x_B)^\alpha=\sum_{|\gamma|+|\beta| =2N \atop \beta\le \alpha} \f{c_{\gamma,\beta}}{x^\gamma} \partial^\beta (x-x_B)^\alpha,
\]
where $c_{\gamma,\beta}$ are constants.

Since we have $|x|\ge \rho(x_B)$ for $x\in B$ and $r_B<\rho(x_B)$, we further imply, for $x\in B$,
\[
\begin{aligned}
	|\Delta_\nu^N(x-x_B)^\alpha|&\lesi  \sum_{|\gamma|+|\beta| =2N\atop \beta\le \alpha} \f{1}{\rho(x_B)^{|\gamma|}}|\partial^\beta (x-x_B)^\alpha|\\
	&\lesi  \sum_{|\gamma|+|\beta| =2N\atop \beta\le \alpha} \f{r_B^{|\alpha-\beta|}}{\rho(x_B)^{|\gamma|}}\\
	&\lesi     \sum_{|\gamma|+|\beta| =2N\atop \beta\le \alpha} \f{1}{\rho(x_B)^{|\gamma|+|\beta|-|\alpha|}}\\
	&\sim \f{1}{\rho(x_B)^{2N-|\alpha|}}.
\end{aligned}
\]
Hence,
\[
\begin{aligned}
	\Big|\int (x-x_B)^\alpha a(x)\,dx\Big|&\lesi \f{1}{\rho(x_B)^{2N-|\alpha|}}\|b\|_{1}\\
	&\lesi |B|^{1-1/p} \f{r_B^{2N}}{\rho(x_B)^{2N-|\alpha|}} = |B|^{1-1/p} r_B^{|\alpha|} \Big(\f{r_B}{\rho_B}\Big)^{2N-|\alpha|}. 
\end{aligned}
\]
This confirms \eqref{eq-integral of a}.

We now turn to prove 	$a\in H^p_{\rho}(\Rn_+)$ as long as $r_B<\rho(x_B)$. Recall that $S_0(B)=B, S_j(B)=2^{j}B\setminus 2^{j-1}B, j\ge 1$. Set $\omega = \lfloor n(1/p-1)\rfloor$. Let $\mathcal{V}_j$ be the span of the polynomials $\big\{(x-x_B)^\alpha\big\}_{|\alpha|\le \om }$ on $S_j(B)$ corresponding inner product space given by
$$ \ip{f,g}_{\mathcal{V}_j}:=\fint_{S_j(B)} f(x)g(x)\,dx.$$
Let $\{u_{j,\alpha}\}_{|\alpha|\le \om }$ be an orthonormal basis for $\mathcal{V}_j$ obtained via the Gram--Schmidt process applied to $\big\{(x-x_B)^\alpha\big\}_{|\alpha|\le \om }$
which, through homogeneity and uniqueness of the process, gives
\begin{align}\label{eq:mol1}
	u_{j,\alpha}(x)=\sum_{|\beta|\le \om } \lambda_{\alpha,\beta}^j (x-x_B)^\beta,
\end{align}
where for each $|\alpha|, |\beta|\le \om $ we have
\begin{align}\label{eq:mol2}
	|u_{j,\alpha}(x)|\le C\qquad\text{and}\qquad |\lambda_{\alpha,\beta}^j |\lesi (2^j r_B)^{-|\beta|}. 
\end{align}
Let $\{v_{j,\alpha}\}_{|\alpha|\le \om }$ be the dual basis of $\big\{(x-x_B)^\alpha\big\}_{|\alpha|\le \om }$ in $\mathcal{V}_j$; 
that is, it is the unique collection of polynomials such that
\begin{align}\label{eq:mol3} \ip{v_{j,\alpha}, (\cdot-x_B)^\beta}_{\mathcal V_j}=\delta_{\alpha,\beta}, \qquad  |\alpha|, |\beta|\le \om . \end{align}
Then we have
\begin{align}\label{eq:mol5}
	\Vert v_{j,\alpha}\Vert_\infty \lesi (2^jr_B)^{-|\alpha|},\qquad \forall \; |\alpha|\le \om .
\end{align}
Now let $P:={\rm proj}_{\mathcal{V}_0} (a)$ be the orthogonal projection of $a$ onto $\mathcal{V}_0$. Then  we have
\begin{align}\label{eq:mol6}
	P = \sum_{|\alpha|\le \om } \ip{a, u_{0,\alpha}}_{\mathcal V_0} u_{0,\alpha} = \sum_{|\alpha|\le \om } \ip{a,(\cdot-x_B)^\alpha}_{\mathcal V_0} v_{0,\alpha}.
\end{align}
Let $j_0\in \mathbb N$ such that $2^{j_0}r_B\ge \rho (x_B)>2^{j_0-1}r_B$. Then we write
\[
\begin{aligned}
	P =& \sum_{|\alpha|\le \om } \ip{a,(\cdot-x_B)^\alpha}_{\mathcal V_0} v_{0,\alpha}\\
	&=\sum_{|\alpha|\le \om }\sum_{j=0}^{j_0-2} \ip{a,(\cdot-x_B)^\alpha} \Big[\f{v_{j,\alpha}}{|S_j(B)|}-\f{v_{j+1,\alpha}}{|S_{j+1}(B)|}\Big] + \ip{a,(\cdot-x_B)^\alpha} \f{v_{j_0-1,\alpha}}{|S_{j_0-1}(B)|}.
\end{aligned}
\]
Hence, we can decompose
\[
\begin{aligned}
	a &= (a-P) +\sum_{|\alpha|\le \om }\sum_{j=0}^{j_0-2} \ip{a,(\cdot-x_B)^\alpha} \Big[\f{v_{j,\alpha}}{|S_j(B)|}-\f{v_{j+1,\alpha}}{|S_{j+1}(B)|}\Big] + \sum_{|\alpha|\le \om}\ip{a,(\cdot-x_B)^\alpha}\f{v_{j_0-1,\alpha}}{|S_{j_0-1}(B)|}\\
	&=a_1 +\sum_{|\alpha|\le \om }\sum_{j=0}^{j_0-2}a_{2,j,\alpha} +\sum_{|\alpha|\le \om }a_{3,\alpha}.
\end{aligned}
\]

Let us now outline the important properties of the functions in the above decomposition. 
For $a_1$ we observe that for all $|\alpha|\le \om $,
\begin{align}\label{eq:mol8}
	\supp a_1 \subset B, &&
	\int a_1(x)(x-x_B)^\alpha dx=0, &&
	\Vert a_1\Vert_{L^\vc}\lesi |B|^{-1/p}.
\end{align}
Note that the property 
\[
\int a_1(x)(x-x_B)^\alpha dx=0 \ \text{for all $|\alpha|\le \om$}
\]
implies that 
\[
\int a_1(x)x^\alpha dx=0 \ \text{for all $|\alpha|\le \om$}.
\]
Hence, in this case $a_1$ is a $(p,\rho)$ atom.


Next, for $a_{2,j,\alpha}$, it is obvious that $\supp a_{2,j,\alpha} \subset 2^{j+1}B$. In addition, from \eqref{eq:mol3}, we have
\[
\int a_{2,j,\alpha}(x)(x-x_B)^\beta dx=0 \ \text{for all $|\beta|\le \om $},
\]
which implies
\[
\int a_{2,j,\alpha}(x)x^\beta dx=0 \ \text{for all $|\beta|\le \om $}.
\]
We now estimate the size of $a_{2,j,\alpha}$. Using \eqref{eq-integral of a} to write
\[
\begin{aligned}
	\|a_{2,j,\alpha}\|_{\vc}&\lesi (2^jr_B)^{-|\alpha|} |2^jB|^{-1} \Big|\int a(x)(x-x_B)^\alpha dx\Big| \\
	&\lesi |2^jB|^{-1 -|\alpha|/n} |B|^{1-1/p}r_B^{|\alpha|}\Big(\f{r_B}{\rho (x_B)}\Big)^{2N-|\alpha|}\\
	&\lesi 2^{-j(2N+n-n/p)} |2^jB|^{-1/p}.
\end{aligned}
\]
This means that $a_{2,j,\alpha}$ is a multiple of a $(p,\rho)$ atom, which further implies
\[
\Big\|\sum_{|\alpha|\le \om }\sum_{j=0}^{j_0-2}a_{2,j,\alpha}\Big\|_{H^p_{\rho}}(\Rn_+)\lesi 1,
\]
as long as $N>n(1/p-1)$.

Next, for $a_{3,\alpha}$ we first have $\supp a_{3,\alpha} \subset 2^{j_0}B$. Moreover,  using \eqref{eq-integral of a} again,
\[
\begin{aligned}
	\|a_{2,j,\alpha}\|_{\vc}&\lesi (2^{j_0}r_B)^{-|\alpha|} |2^{j_0}B|^{-1} \Big|\int a(x)(x-x_B)^\alpha dx\Big| \\
	&\lesi |2^{j_0}B|^{-1 -|\alpha|/n} |B|^{1-1/p}r_B^{|\alpha|}\Big(\f{r_B}{\rho (x_B))}\Big)^{2N-|\alpha|}\\
	&\lesi 2^{-j_0(2N+n-n/p)} |2^{j_0}B|^{-1/p}\\
	&\lesi |2^{j_0}B|^{-1/p},
\end{aligned}
\]
as long as $N>n(1/p-1)$.

It follows that $a_{3,\alpha}$ is a $(p,\rho)$ atom and $\|a_{3,\alpha}\|_{H^p_\rho(\Rn_+)}\lesi 1$.

This completes our proof.
\end{proof}

\bigskip

\subsection{Campanato spaces associated to the Laguerre operator $\Delta_\nu$} 

The proof of Theorem \ref{mainthm-dual} is similarly to those in \cite{BD} and hence we just sketch out the main ideas.

\begin{lem}\label{lem-covering lemma}
	Let $\nu\in(-1/2,\vc)^n$. There exist a family of balls $\{B(x_\xi,\rho(x_\xi)): \xi\in \mathcal I\}$ and a family of functions $\{\psi_\xi: \xi \in \mathcal I\}$ such that 
	\begin{enumerate}[{\rm (i)}]
		\item $\displaystyle \bigcup_{\xi\in \mathcal I}B(x_\xi,\rho(x_\xi)) = \mathbb R^n_+$;
		\item $\{B(x_\xi,\rho(x_\xi)/5): \xi\in \mathcal I\}$ is pairwise disjoint;
		\item  $\displaystyle \sum_{\xi\in \mathcal I} \chi_{B(x_\xi,\rho(x_\xi))}\lesi 1$;
		\item $\supp \psi_\xi\subset B(x_\xi,\rho(x_\xi))$ and $0\le \psi_\xi \le 1$ for each $\xi \in \mathcal I$;
		\item $\displaystyle \sum_{\xi\in \mathcal I} \psi_\xi =1$.
	\end{enumerate}
\end{lem}	
\begin{proof}
	Consider the family $\{B(x,\rho(x)/5): x\in \mathbb R^n_+\}$. Since $\rho(x)\le 1$ for every $x\in \mathbb R^n_+$, by Vitali's covering lemma we can extract a sub-family denoted by $\{B(x_\xi,\rho(x_\xi)/5): \xi\in \mathcal I\}$ satisfying (i) and (ii).
	
	The item (iii) follows directly from (ii) and \eqref{eq- critical function}.
	
	For each $\xi\in \mathcal I$, define
	\[
	\displaystyle \psi_\xi(x) = \begin{cases}
		\displaystyle \f{\chi_{B(x_\xi,\rho(x_\xi))}(x)}{\sum_{\theta\in \mathcal I}\chi_{B(x_\theta,\rho(x_\theta))}(x)}, \ \ & x\in B(x_\xi,\rho(x_\xi)),\\
		0, & x \notin B(x_\xi,\rho(x_\xi)).		
	\end{cases}
	\] 
	Then $\{\psi_\xi\}_{\xi\in \mathcal I}$ satisfies (iv) and (v).
	
	This completes our proof.
\end{proof}

Since  $\Delta_\nu$ is a non-negative self-adjoint operator  satisfying the Gaussian upper bound (see Proposition \ref{prop- delta k pt d>2}), it is well-known that the kernel $K_{\cos(t\sqrt{\Delta_\nu})}(\cdot,\cdot)$ of $\cos(t\sqrt{\Delta_\nu})$ satisfies 
\begin{equation}\label{finitepropagation}
	{\rm supp}\,K_{\cos(t\sqrt{\Delta_\nu})}(\cdot,\cdot)\subset \{(x,y)\in \Rn_+\times \Rn_+:
	|x-y|\leq t\}.
\end{equation}
See for example \cite{Sikora}.

As a consequence of \cite[Lemma 3]{Sikora}, we have:
\begin{lem}\label{lem:finite propagation}
	Let $\nu\in (-1/2,\vc)^n$. Let $\varphi\in C^\vc_0(\mathbb{R})$ be an even function with {\rm supp}\,$\varphi\subset (-1, 1)$ and $\displaystyle \int \varphi =2\pi$. Denote by $\Phi$ the Fourier transform of $\varphi$.  Then for any $k\in \mathbb N$ the kernel $K_{(t^2\Delta_\nu)^k\Phi(t\sqrt{\Delta_\nu})}$ of $(t^2\Delta_\nu)^k\Phi(t\sqrt{\Delta_\nu})$ satisfies 
	\begin{equation}\label{eq1-lemPsiL}
		\displaystyle
		{\rm supp}\,K_{(t^2\Delta_\nu)^k\Phi(t\sqrt{\Delta_\nu})}\subset \{(x,y)\in \Rn_+\times \Rn_+:
		|x-y|\leq t\},
	\end{equation}
	and
	\begin{equation}\label{eq2-lemPsiL}
		|K_{(t^2\Delta_\nu)^k\Phi(t\sqrt{\Delta_\nu})}(x,y)|\lesi  \f{1}{t^n}
	\end{equation}
	for all $x,y \in \Rn_+$ and $t>0$.
\end{lem}

Lemmas \ref{03} and \ref{07} below can be proved similarly to Lemmas 4.12 and 4.14 in \cite{BD} and we omit the details. 

\begin{lem} \label{03}
Let $s\ge 0$, $M\in \mathbb N, M\ge \lfloor n(1/p-1)\rfloor$, $\nu\in (-1/2,\vc)^n$ and $\rho$ be as in \eqref{eq- critical function}. Let $\Phi$ be as in Lemma \ref{lem:finite propagation}. Then for $k>\f{s}{2}$, there exists $C>0$ such that for all $f\in BMO^{s,M}_{\rho}(\mathbb{R}^n_+)$,
\begin{equation}
	\label{eq1-Carleson}
	\sup_{B: \, {\rm balls}} \f{1}{|B|^{2s/n +1}} \int_0^{r_B}\int_B |(t^2\Delta_\nu)^k  \Phi(t\sqrt{\Delta_\nu})f(x)|^2\f{dxdt}{t}\le  C\|f\|_{BMO^{s,M}_{\rho}(\mathbb{R}^n_+)}.
\end{equation}
\end{lem}

\begin{lem} \label{07}
Let $\nu\in (-1/2,\vc)^n$, $\rho$ be as in \eqref{eq- critical function} and $\Phi$ be as in Lemma \ref{lem:finite propagation}. Let  $\f{n}{n+\gamma_\nu}<p\le 1$, $s=n(1/p-1)$ and $M\in \mathbb N, M\ge \lfloor s\rfloor$, where $\gamma_\nu=\nu_{\min}+1/2$. 
Then  for every $f \in BMO^{s,M}_{\rho} (\mathbb{R}^n_+)$ and every $(p,\rho)$-atom $a$,
\begin{align} \label{abc}
	\int_{\mathbb R^n_+} f (x) {a(x)}dx  =  C_k\int_{\mathbb R^n_+ \times (0,\infty)}(t^2\Delta_\nu )^k\Phi(t\sqrt{\Delta_\nu})f(x) { t^2\Delta_\nu e^{-t^2\Delta_\nu}a(x)} \frac{dx dt}{t},
\end{align}
where $C_k = \Big[\displaystyle \int_0^\vc z^{k}\Phi(\sqrt z)e^{-z}  dz\Big]^{-1}$. 
\end{lem}

We are ready to give the proof of Theorem \ref{mainthm-dual}.

\begin{proof} [Proof of  Theorem \ref{mainthm-dual}:] Fix $\f{n}{n+\gamma_\nu}<p\le 1$ and $s=n(1/p-1)$. We divide the proof into several steps.

\medskip

\noindent {\bf Step 1.} \underline{Proof of  $BMO^{s,M}_{\rho}(\mathbb{R}^n_+) \subset (H^p_{\Delta_\nu}(\mathbb{R}^n_+))^\ast$.}

\medskip
Let $f \in BMO^{s,M}_{\rho}(\mathbb{R}^n_+)$ and let $a$ be a  $(p,\rho)$-atoms. Then
by Lemma \ref{07} and Proposition 3.2 in \cite{Yan} (see also \cite{CMS}),
\begin{align*}
	\left|\int_{\Rn_+} f(x) {a(x)}dx  \right|
	& =  \left| \int_{\Rn_+ \times (0,\infty)}(t^2\Delta_\nu)^k\Phi(t\sqrt{\Delta_\nu}) f(x) { t^2\Delta_\nu e^{-t^2\Delta_\nu}a(x)} \frac{dx dt}{t} \right| \\
	&\leq \sup_{B}\left( \f{1}{|B|^{2s/n +1}} \int_0^{r_B}\int_B |(t^2\Delta_\nu)^k\Phi(t\sqrt{\Delta_\nu})f(x)|^2\f{dx dt}{t}\right)^{1/2} \\
	&\qquad \times \left\| \left( \int_0^\vc \int_{|x-y|<t}|t^2\Delta_\nu e^{-t^2\Delta_\nu}a(y)|^2 \frac{dy dt}{t^{n+1}}\right)^{1/2}\right\|_{L^p(\mathbb{R}^n_+)}.
\end{align*}
Using Lemma \ref{03},
\[
\sup_{B}\left( \f{1}{|B|^{2s/n +1}} \int_0^{r_B}\int_B |(t^2\Delta_\nu)^k\Phi(t\sqrt{\Delta_\nu})f(x)|^2\f{dx dt}{t}\right)^{1/2}\lesi \|f\|_{BMO^{s,M}_{\rho}(\mathbb{R}^n_+)}.
\]
In addition, by Theorem \ref{mainthm2s} and Theorem 1.3 in \cite{SY}, we have
\[
\left\| \left( \iint_{\Gamma(x)}|t^2\Delta_\nu e^{-t^2\Delta_\nu}a(y)|^2 \frac{dy dt}{t^{n+1}}\right)^{1/2}\right\|_{L^p(\mathbb{R}^n_+)}\lesi 1.
\]
Consequently,
\[
	\left|\int_{\Rn_+} f(x) {a(x)}dx  \right|\lesi \|f\|_{BMO^{s,M}_{\rho}(\mathbb{R}^n_+)}
\]
for $f\in BMO^{s,M}_{\rho}(\mathbb{R}^n_+)$ and all $(p,\rho)$ atoms $a$.

It follows that  $BMO^{s,M}_{\rho}(\mathbb{R}^n_+) \subset (H^p_{\Delta_\nu}(\mathbb{R}^n_+))^\ast$.

\medskip 

\noindent 	{\bf Step 2. } \underline{Proof of  $(H^p_{\Delta_\nu}(\mathbb{R}^n_+))^\ast \subset BMO^{s,M}_{\rho}(\mathbb{R}^n_+) $.}

\medskip

Let $\{\psi_\xi\}_{\xi \in \mathcal{I}}$ and $\{B(x_\xi,\rho(x_{\xi}))\}_{\xi \in \mathcal{I}}$ as in Corollary \ref{cor1}. 
Set $B_\xi:= B(x_\xi, \rho(x_\xi))$ and we will claim that for any $f \in L^2(\mathbb{R}^n_+)$ and $\xi \in \mathcal{I}$, we have $\psi_\xi f \in H^p_{\Delta_\nu} (\Rn_+)$ and
\begin{align} \label{eq:L2}
	\|\psi_{\xi}f\|_{H^p_{\Delta_\nu}(\mathbb{R}^n_+)} \leq C \left|B_\xi\right|^{\frac{1}{p}-\frac{1}{2}} \|f\|_{L^2(\mathbb{R}^n_+)}.
\end{align}
It suffices to prove that 
\[
\left\|\sup_{t>0}\big|e ^{-t^2 \Delta_\nu}(\psi_{\xi}f)\big| \right\|_{p}\lesssim  |B_\xi|^{\frac{1}{p}-\frac{1}{2}}\|f\|_{2}.
\]
Indeed,  by H\"{o}lder's inequality,
\begin{equation} \label{eq:B}
	\begin{split}
		\left\|\sup_{t>0}\big|e ^{-t^2 \Delta_\nu}(\psi_{\xi}f)\big| \right\|_{L^p(4B_\xi)}& \leq |4B_\xi|^{\frac{1}{p}-\frac{1}{2}}
		\left\| \sup_{t>0}\big|e ^{-t^2 \Delta_\nu}(\psi_{\xi}f)\big| \right\|_{L^2(4B_\xi)}  \\
		& \lesssim  |B_\xi|^{\frac{1}{p}-\frac{1}{2}}\|f\|_{L^2(\mathbb{R}^n_+)}.
	\end{split}
\end{equation}
If $x \in \mathbb R^n_+\backslash 4B_\xi$, then  applying Proposition \ref{prop- delta k pt d>2}  and H\"older's inequality, we get
\begin{align*}
	\big|e ^{-t^2 \Delta_\nu}(\psi_{\xi}f)(x)\big|
	& \lesi \int_{B_\xi} \left( \frac{t}{\rho(y)}\right)^{-\gamma_\nu} \frac{1}{t^n} \exp\left( -\frac{|x-y|^2}{ct^2}\right)|f(y)|dy \\
	& \sim \int_{B_\xi} \left(  \frac{t}{\rho(x_\xi)}\right)^{-\gamma_\nu} \frac{1}{t^n} \exp\left( -\frac{|x-y|^2}{ct^2}\right)|f(y)|dy \\
	& \lesi \frac{\rho(x_\xi)^{\gamma_\nu}}{|x-x_\xi|^{n+\gamma_\nu}}\int_{B_\xi}  |f(y)|dy \\			
	&\lesi \frac{\rho(x_\xi)^{\gamma_\nu}}{|x-x_\xi|^{n+\gamma_\nu}} |B_\xi|^{\frac{1}{2}} \|f\|_{L^2(\mathbb{R}^n_+)},
\end{align*}
which implies that
\begin{equation} \label{eq:BC}
	\begin{split}
		\left\|\sup_{t>0}\big|e ^{-t^2 \Delta_\nu}(\psi_{\xi}f)\big| \right\|_{L^p(\Rn_+ \backslash 4B_\xi)} \lesi  |B_\xi|^{\frac{1}{p}-\frac{1}{2}}\|f\|_{L^2(\mathbb{R}^n_+)},
	\end{split}
\end{equation}
as long as $\f{n}{n+\gamma_\nu}<p\le 1$.

Combining \eqref{eq:B} and \eqref{eq:BC} yields \eqref{eq:L2}.

Assume that $\ell \in (H^p_{\Delta_\nu}(\mathbb{R}^n_+))^\ast$.
For each index $\xi \in \mathcal{I}$ we define
\begin{align*}
	\ell_{\xi} f := \ell(\psi_{\xi} f), \quad f \in L^2(\mathbb{R}^n_+).
\end{align*}
By \eqref{eq:L2},
\begin{align*}
	|\ell_{\xi}(f)| \leq C \|\psi_{\xi}f\|_{H^p_{\Delta_\nu}(\mathbb{R}^n_+)} \leq C|B_\xi|^{\frac{1}{p}-\frac{1}{2}}\|f\|_{L^2(\mathbb{R}^n_+)}.
\end{align*}
Hence there exists $g_{\xi} \in L^{2}(B_\xi)$ such that
\begin{align*}
	\ell_{\xi}(f) = \int_{B_\xi} f(x)g_{\xi}(x)dx ,
	\quad f \in L^2 (\mathbb{R}^n_+).
\end{align*}

We define $g = \sum_{\xi \in \mathcal{I}}1_{B_{\xi}} g_{\xi}$. Then, if $f =\sum_{i=1}^k \lambda_j a_j$, where $k \in \mathbb{N}$,
$\lambda_i \in \mathbb{C}$, and $a_i$ is a $(p,2,\rho)$-atom, $i =1, \cdots, k$, we have
\begin{align*}
	\ell(f) = \sum_{i=1}^k \lambda_i  \ell (a_i)
	&=  \sum_{i=1}^k \lambda_i \sum_{\xi \in \mathcal{I}}\ell ( \psi_\xi a_i) =  \sum_{i=1}^k \lambda_i \sum_{\xi \in \mathcal{I}}\ell_\xi (a_i) \\
	&= \sum_{i=1}^k \lambda_i \sum_{\xi \in \mathcal{I}} \int_{B_\xi} a_i(x)g_\xi (x)dx \\
	&= \sum_{i=1}^k \lambda_i \int_{\Rn_+} g(x)a_i (x)dx \\
	&= \int_{{\Rn_+}} f(x)g(x)dx .
\end{align*}

Suppose that $B = B(x_B, r_B) \in \mathbb R^n_+$ with $r_B < \rho(x_B)$, and $0 \not\equiv f \in L_0^2(B)$, that is,
$f \in L^2(\mathbb{R}^n_+)$ such that $\supp f \subset B$ and $\displaystyle \int_B x^\alpha f(x)dx  =0$ with all $|\alpha|\le M$. Then similarly to \eqref{eq:L2},
\[
\|f\|_{H^p_{\Delta_\nu}(\mathbb{R}^n_+) }\lesi  \|f\|_{L^2} |B|^{1/p -1/2}.
\]

Hence
\begin{align*}
	|\ell(f)| = \Big| \int_B fg\Big| \leq \|\ell\|_{(H^p_{\Delta_\nu}(\mathbb{R}^n_+))^\ast} \|f\|_{H^p_{\Delta_\nu}(\mathbb{R}^n_+) } \leq C \|\ell\|_{(H^p_{\Delta_\nu}(\mathbb{R}^n_+))^\ast} \|f\|_{L^2} |B|^{1/p -1/2}.
\end{align*}
From this we conclude that $g \in (L_0^2(B))^\ast$ and
\[
\|g\|_{(L_0^2(B))^\ast} \leq C \|\ell\|_{(H^p_{\Delta_\nu}(\mathbb{R}^n_+))^\ast}|B|^{1/p -1/2}.
\]
But by elementary functional analysis (see Folland and Stein \cite[p. 145]{FS}),
\[
\|g\|_{(L_0^2(B))^\ast}=\inf_{P\in \mathcal P_M} \|g-P\|_{L^{2}(B)},
\]
where $P \in \mathcal{P}_M$ is the set of all polynomials  of degree at most $M$.
Hence
\begin{align} \label{31}
	\sup_{\substack{B: ball \\ r_B <\rho(x_B)}}|B|^{1/2-1/p} \inf_{P\in \mathcal P_M} \|g-P\|_{L^{2}(B)} \leq C \|\ell\|_{(H^p_{\Delta_\nu}(\mathbb{R}^n_+))^\ast} .
\end{align}
Moreover, if $B$ is a ball with $r_B \ge \rho(x_B)$, and $f \in L^2(\mathbb{R}^n_+)$ such that $f \not\equiv 0$
and $\supp f \subset B$, then similarly to \eqref{eq:L2},
\[
\|f\|_{H^p_{\Delta_\nu}(\mathbb{R}^n_+) }\lesi  \|f\|_{L^2} |B|^{1/p -1/2}.
\] 
Hence,
\begin{align*}
	|\ell(f)| = \left|\int_{\Rn_+} fg \right| \leq C\|\ell\|_{(H^p_{\Delta_\nu}(\mathbb{R}^n_+))^\ast } \|f\|_{H^p_{\Delta_\nu}(\mathbb{R}^n_+)}
	\leq   C\|\ell\|_{(H^p_{\Delta_\nu}(\mathbb{R}^n_+))^\ast } |B|^{1/p-1/2} \|f\|_{L^2}.
\end{align*}
Hence
\begin{align} \label{32}
	\sup_{\substack{B: {\rm ball} \\ r_B \ge \rho(x_B)}}|B|^{1/2 -1/p} \|g\|_{L^{2}(B)} \leq C \|\ell\|_{(H^p_{\Delta_\nu}(\mathbb{R}^n_+))^\ast} .
\end{align}
From \eqref{31} and \eqref{32} it follows that $g \in BMO^{s,M}_{\rho } (\mathbb{R}^n_+)$ and
\begin{align*}
	\|g\|_{BMO^{s,M}_{\rho } (\mathbb{R}^n_+)} \leq C  \|\ell\|_{(H^p_{\Delta_\nu}(\mathbb{R}^n_+))^\ast}, \quad \mbox{where } s =n(1/p-1).
\end{align*}

This completes our proof.
\end{proof}

\section{Boundedness of Riesz transforms on Hardy spaces and Campanato spaces associated to $\Delta_\nu$}

In this section, we will study the boundedness of the higher-order Riesz transforms. We first show in Theorem \ref{thm-Riesz transform} below that the higher-order Riesz transforms are Calder\'on-Zygmund operators. Then, in Theorem \ref{thm- boundedness on Hardy and BMO} we will show that the  higher-order Riesz transforms are bounded on our new Hardy spaces and new BMO type spaces defined in Section 4. 

We first give a formal definition of $\Delta_\nu^{-s}$ for any $s>0$. For $s>0$, by the spectral theory we define

\[
\Delta_\nu^{-s} = \int_0^\infty \lambda^{-s} \, dE(\lambda),
\]
where $E(\lambda)$ is the spectral decomposition of $\Delta_\nu$. 

The domain of \( \Delta_\nu^{-s} \) consists of all \( f \in L^2(\Rn_+) \) such that the integral

\[
\int_0^\infty \lambda^{-2s} \, d\langle E(\lambda)f, f \rangle
\]
is finite.

We will show that

\begin{equation}\label{eq-sub formula}
\Delta_\nu^{-s} = \frac{1}{\Gamma(s)} \int_0^\infty u^{s-1}e^{-u\Delta_\nu}  \, du.
\end{equation}
Indeed, by the spectral theorem, the semigroup \( e^{-u\Delta_\nu} \) can be written  as
\[
e^{-u\Delta_\nu} = \int_0^\infty e^{-u\lambda} \, dE(\lambda).
\]

Substituting this into the integral definition of \( \Delta_\nu^{-s} \),
\[
\frac{1}{\Gamma(s)} \int_0^\infty e^{-u\Delta_\nu} u^{s-1} \, du = \frac{1}{\Gamma(s)} \int_0^\infty \left( \int_0^\infty e^{-u\lambda} \, dE(\lambda) \right) u^{s-1} \, du.
\]

Interchanging the order of integration,
\[
\frac{1}{\Gamma(s)} \int_0^\infty e^{-u\Delta_\nu} u^{s-1} \, du = \frac{1}{\Gamma(s)} \int_0^\infty \left( \int_0^\infty e^{-u\lambda} \,  u^{s-1} \, du\right) \, dE(\lambda) .
\]

This inner integral is a standard  {Laplace transform}:

\[
\int_0^\infty e^{-u\lambda} u^{s-1} \, du = \lambda^{-s} \Gamma(s), \quad \text{for } \lambda > 0.
\]

Thus, we obtain 
\[
\begin{aligned}
	\frac{1}{\Gamma(s)} \int_0^\infty e^{-u\Delta_\nu} u^{s-1} \, du &=   \int_0^\infty \lambda^{-s} dE(\lambda)\\
	&=\Delta_\nu^{-s}.
\end{aligned}
\]
This ensures the formula \eqref{eq-sub formula}.

\begin{thm}\label{thm- boudnedness of Riesz nu large}
	Let $\nu\in (-1,\vc)^n$ and $k\in \mathbb N^n$. Then for any $\alpha\in \mathbb N^n$, the operator $\delta_\nu^k \Delta_\nu^{-|k|/2}-\delta_{\nu+\alpha}^k \Delta_{\nu+\alpha}^{-|k|/2}$ is bounded on $L^p(\Rn_+)$ for all $1<p<\vc$. 
\end{thm}
\begin{proof}
	By induction, it suffices to prove for the case $\alpha = e_j$ for all $j=1,\ldots, n$. We will prove for the case $\alpha=e_1$ since other cases can be done similarly.
	
	Using \eqref{eq-sub formula}, we have
	\[
	\begin{aligned}
		\delta_\nu^k \Delta_\nu^{-|k|/2}-\delta_{\nu+e_1}^k \Delta_{\nu+e_1}^{-|k|/2} &=\f{1}{\Gamma(|k|/2)}\int_0^\vc t^{|k|/2}[\delta_\nu^k e^{-t\Delta_\nu}-\delta_{\nu+e_1}^k e^{-t\Delta_{\nu+e_1}}]\f{dt}{t}\\
		&=\f{1}{\Gamma(|k|/2)}\int_0^\vc t^{k_1/2}[\delta_{\nu_1}^{k_1} e^{-t\Delta_{\nu_1}}-\delta_{\nu_1+e_1}^k e^{-t\Delta_{\nu_1+e_1}}]\prod_{j=2}^n t^{k_j/2}\delta_{\nu_j}^{k_j} e^{-t\Delta_{\nu_j}} \f{dt}{t}.
	\end{aligned}
	\]
	Due to Theorem \ref{thm-delta k pt nu}, we have
	\[
	\sup_{t>0}|t^{k_j/2}\delta_{\nu_j}^{k_j} e^{-t\Delta_{\nu_j}}f| \lesi \mathcal Mf
	\]
	and hence the operator $f\mapsto \sup_{t>0}|t^{k_j/2}\delta_{\nu_j}^{k_j} e^{-t\Delta_{\nu_j}}f|$ is bounded on $L^p(\mathbb R_+)$ for each $j=2,\ldots, n$ and for $1<p<\vc$.
	
	On the other hand, from Proposition \ref{prop- difference Riesz}, the operator 
	\[
	f\mapsto \int_0^\vc t^{k_1/2}\big|[\delta_{\nu_1}^{k_1} e^{-t\Delta_{\nu_1}}-\delta_{\nu_1+e_1}^k e^{-t\Delta_{\nu_1+e_1}}]f\big|\f{dt}{t}
	\]
	is bounded on $L^p(\mathbb R_+)$ for $1<p<\vc$.
	
	Consequently, the operator $\delta_\nu^k \Delta_\nu^{-|k|/2}-\delta_{\nu+e_1}^k \Delta_{\nu+e_1}^{-|k|/2}$ is bounded on $L^p(\mathbb R_+)$ for $1<p<\vc$.
	
	This completes our proof.
\end{proof}

We are ready to give the proof of Theorem \ref{thm-Riesz transform}.
\begin{proof}[Proof of Theorem \ref{thm-Riesz transform}:] 

\medskip

 {\textit{We first prove that the higher Riesz $\delta_\nu^k \Delta_\nu^{-|k|/2}$ is bounded on $L^2(\Rn_+)$ by induction.}}
	
	For $|k|=1$, from \eqref{eq-Delta nu},
	\[
	\begin{aligned}
		\|\Delta_\nu^{1/2}f\|_2^2&=\langle \Delta_\nu^{1/2}f, \Delta_\nu^{1/2}f \rangle =\langle \Delta_\nu f,  f \rangle\\
		&=\sum_{j=1}^n\langle \delta_{\nu_j}^* \delta_{\nu_j}f,  f \rangle = \sum_{j=1}^n\langle  \delta_{\nu_j}f,  \delta_{\nu_j}f \rangle\\
		&=\sum_{j=1}^n\|\delta_{\nu_j}f\|_2^2,
	\end{aligned}
	\]
	which implies that  
	\[
	\|\delta_{\nu_j}f\|_2\le \|\Delta_\nu^{1/2}f\|_2, \ \ j=1,\ldots, n.
	\]
	Hence, the Riesz transform $\delta_\nu^k \Delta_\nu^{-|k|/2}$ is bounded on $L^2(\Rn_+)$ if $|k|=1$. 
	
	Assume that the Riesz transform $\delta_\nu^k \Delta_\nu^{-|k|/2}$ is bounded on $L^2(\Rn_+)$ for all $k$ with $|k|=\ell$ for some $\ell\ge 1$. We need to prove that for any $k$ with $|k|=\ell+1$ the Riesz transform $\delta_\nu^k \Delta_\nu^{-|k|/2}$ is bounded on $L^2(\Rn_+)$.
	
	If $k_i\le 1$ for all $i=1,\ldots, n$, we might assume that $k_1=\ldots=k_j=1$ and $k_{j+1}=\ldots = k_n =0$ for some $1\le j\le n$. Then we can write
	\[
	\delta_\nu^k \Delta_\nu^{-|k|/2} = [\delta_{\nu_1}\Delta_{\nu_1}^{-1/2}\otimes \ldots \otimes \delta_{\nu_j}\Delta_{\nu_j}^{-1/2}\otimes I \otimes \ldots \otimes I]\circ [\Delta_{\nu_1}^{1/2}\otimes \ldots \otimes \Delta_{\nu_j}^{1/2}\otimes I \otimes \ldots \otimes I]\Delta_\nu^{-|k|/2}.
	\]
	Since each $\delta_{\nu_i}\Delta_{\nu_i}^{-1/2}$ is bounded on $L^2(\mathbb R_+)$ for $i=1,\ldots, j$, the operator $\delta_{\nu_1}\Delta_{\nu_1}^{-1/2}\otimes \ldots \otimes \delta_{\nu_j}\Delta_{\nu_j}^{-1/2}\otimes I \otimes \ldots \otimes I$ is bounded on $L^2(\Rn_+)$. On the other hand, by the joint spectral theory, the operator $[\Delta_{\nu_1}^{1/2}\otimes \ldots \otimes \Delta_{\nu_j}^{1/2}\otimes I \otimes \ldots \otimes I]\Delta_\nu^{-|k|/2}$ is bounded on $L^2(\Rn_+)$. Therefore, the Riesz transform $\delta_\nu^k \Delta_\nu^{-|k|/2} $ is bounded on $L^2(\Rn_+)$. Hence, we have prove the $L^2(\Rn_+)$-boundedness for the Riesz transform $\delta_\nu^k \Delta_\nu^{-|k|/2} $ in the case $k_i\le 1$ for all $i=1,\ldots, n$.
		
	Otherwise, we might assume that $k_1\ge 2$. By Theorem \ref{thm- boudnedness of Riesz nu large}, we might assume that $\nu_1\ge k_1 +2$. Then using the fact
	\[
	\delta_{\nu_1}^2 = -\Delta_{\nu_1} +\f{2\nu+1}{x_1}\delta_{\nu_1},
	\] 
	which implies
	\[
	\delta_\nu^k \Delta_\nu^{-|k|/2} = -\delta_\nu^{k-2e_1} \Delta_{\nu_1}\Delta_\nu^{-|k|/2} + (2\nu+1)\delta_\nu^{k-2e_1}\Big[\f{1}{x_1}\delta_{\nu_1}\Delta_\nu^{-|k|/2} \Big].
	\]
	For the first operator, we can write 
	\[
	-\delta_\nu^{k-2e_1} \Delta_\nu^{-(|k|-2)/2}\circ \Delta_{\nu_1}\Delta_\nu^{-1}.
	\]
	The operator $\delta_\nu^{k-2e_1} \Delta_\nu^{-|k|/2}$ is bounded on $L^2$ due to the inductive hypothesis, while the operator $\Delta_{\nu_1}\Delta_\nu^{-1}$ is bounded on $L^2$ due to the joint spectral theory. Hence, the first operator is bounded on $L^2$.
	
	For the second operator, using the product rule we have
	\[
	\delta_\nu^{k-2e_1}\Big[\f{1}{x_1}\delta_{\nu_1}\Delta_\nu^{-|k|/2} \Big]=\sum_{j=0}^{k_1-2}\f{c_j}{x_1^{1+j}}\delta_\nu^{k-(1+j)e_1}\Delta_\nu^{-|k|/2}
	\]
	Hence, it suffices to prove that the operator 
	$$
	f\mapsto \f{1}{x^{j}_1}\delta_{\nu}^{k-je_1}\Delta_{\nu}^{-|k|/2}f
	$$
	is bounded on $L^2(\Rn_+)$ for each $j=1,\ldots, k_1-2$.

	Denote by $\delta_{\nu}^{k-je_1}\Delta_{\nu}^{-|k|/2}(x,y)$ the kernel of $\delta_{\nu}^{k-je_1}\Delta_{\nu}^{-|k|/2}$. By Proposition \ref{prop- delta k pt d>2} we have, for each $j=2,\ldots, k_1-2$ and a fixed $\epsilon \in (0,1)$,
	\[
	\begin{aligned}
		\f{1}{x_1^j}\delta_{\nu}^{k-je_1}\Delta_{\nu}^{-|k|/2}(x,y) &= \f{1}{x_1^j}\int_0^\vc t^{|k|/2}\delta_{\nu}^{k-je_1} p_t^{\nu}(x,y) \f{dt}{t}\\
		&\lesi \int_0^\vc  \Big(\f{\sqrt t}{x_1}\Big)^{j}\f{1 }{t^{n/2}}\exp\Big(-\f{|x-y|^2}{ct}\Big)\Big(1+\f{\sqrt t}{x_1}\Big)^{-(\nu_{1}+3/2)} \f{dt}{t}\\&\lesi \int_0^\vc   \f{\sqrt t}{x_1} \f{1 }{t^{n/2}}\exp\Big(-\f{|x-y|^2}{ct}\Big)\Big(1+\f{\sqrt t}{x_1}\Big)^{-(\nu_{1}-j+5/2)} \f{dt}{t}\\
		&\lesi  \f{1}{x_1|x-y|^{n-1}} \min\Big\{\Big(\f{|x-y|}{ x_1}\Big)^{-\epsilon}, \Big(\f{|x-y|}{ x_1}\Big)^{-2}\Big\},
	\end{aligned}
	\]
	since $\nu_1-j+5/2\ge 2$ for $j=2,\ldots, k_1-2$.

	Hence,
	\[
	\begin{aligned}
		\Big|\f{1}{x_1^j}\delta_{\nu}^{k-je_1}\Delta_{\nu}^{-|k|/2}f(y)\Big|&\lesi \int_{|x-y|\le x_1}  \f{1}{x_1|x-y|^{n-1}} \Big(\f{|x-y|}{x_1}\Big)^{-\epsilon} |f(y)|dy\\
		& \ \ \ \ +\int_{|x-y|> x_1} \f{1}{x_1|x-y|^{n-1}} \Big(\f{|x-y|}{x_1}\Big)^{-2} |f(y)|dy  \\
		&=:I_1 +I_2.
	\end{aligned}
	\]
	It is easy to see that 
	\[
	\begin{aligned}
		I_1&\lesi \int_{|x-y|\le   x_1}\Big(\f{|x-y|}{x_1}\Big)^{1-\epsilon}\f{|f(y)|}{|x-y|^{n}}dy\\
		&\lesi \mathcal Mf(x)
	\end{aligned}
	\]
	and
	\[
	\begin{aligned}
		I_2&\lesi \int_{|x-y|> x_1} \f{|f(y)|}{|x-y|^{n}}\Big(\f{x_1}{|x-y|}\Big)^{\nu_1+1/2} dy\\
		&\lesi \mathcal Mf(x),
	\end{aligned}
	\]
	where $\mathcal M$ is the Hardy-Littlewood maximal function.

	It follows that the operator $\f{1}{x_1^j}\delta_{\nu}^{k-je_1}\Delta_{\nu}^{-|k|/2}$ is bounded on $L^2(\Rn_+)$, which completes the proof of the $L^2(\Rn_+)$-boundedness of $\delta_\nu^k\Delta_\nu^{-|k|/2}$.

\bigskip

{\textit{It remains to prove the kernel estimates for  the Riesz transforms $\delta_\nu^k\Delta_\nu^{-|k|/2}$.}
}

	Recall that $\delta^k_\nu \Delta_\nu^{-|k|/2}(x,y)$ is the kernel of $\delta^k_\nu \Delta_\nu^{-|k|/2}$. By \eqref{eq-sub formula} and Proposition \ref{prop- delta k pt d>2}, we have
	\begin{equation}\label{eq- R kernel 1}
		\begin{aligned}
			|\delta^k_\nu \Delta_\nu^{-|k|/2}(x,y)| &= c\Big|\int_0^\vc t^{|k|/2}\delta^k_\nu p_t^\nu(x,y) \f{dt}{t}\Big|\\
			&\lesi \int_0^\vc  \f{1 }{t^{n/2}}\exp\Big(-\f{|x-y|^2}{ct}\Big)\Big(1+\f{\sqrt t}{\rho(x)}+\f{\sqrt t}{\rho(y)}\Big)^{-(\nu_{\min}+1/2)} \f{dt}{t}\\
			&\lesi  \f{1}{|x-y|^n} \Big(1+\f{|x-y|}{\rho(x)}+\f{|x-y|}{\rho(y)}\Big)^{-(\nu_{\min}+1/2)},
		\end{aligned}
	\end{equation}
	which implies
	\[
	| \delta^k_\nu \Delta_\nu^{-|k|/2}(x,y)|\lesi \f{1}{|x-y|^n}, \ \ \ x\ne y.
	\]
	
	We will show that 
	\[
	| \delta^k_\nu \Delta_\nu^{-|k|/2}(x,y)-\delta^k_\nu \Delta_\nu^{-|k|/2}(x,y')| \lesi \f{1}{|x-y|^n} \Big(\f{|y-y'|}{|x-y|}\Big)^{\nu_{\min}+1/2},
	\]		
	whenever $|y-y'|\le\f{1}{2}|x-y|$.
	
	Indeed, if $|y-y'|\ge \max\{\rho(y),\rho(y')\}$, then from \eqref{eq- R kernel 1} we have
	\[
	\begin{aligned}
		| \delta^k_\nu \Delta_\nu^{-|k|/2}(x,y)-\delta^k_\nu \Delta_\nu^{-|k|/2}(x,y')|&\lesi | \delta^k_\nu \Delta_\nu^{-|k|/2}(x,y)|+|\delta^k_\nu \Delta_\nu^{-|k|/2}(x,y')| \\
		&\lesi \f{1}{|x-y|^n}\Big[\Big(\f{\rho(y)}{|x-y|}\Big)^{\nu_{\min}+1/2}+\Big(\f{\rho(y')}{|x-y|}\Big)^{\nu_{\min}+1/2}\Big]\\
		&\lesi \f{1}{|x-y|^n} \Big(\f{|y-y'|}{|x-y|}\Big)^{\nu_{\min}+1/2}
	\end{aligned}
	\]
	
	If $|y-y'|\lesi \max\{\rho(y),\rho(y')\}$, then by  the mean value theorem,
	\begin{equation}\label{eq- R(x,y)-R(x,y')}
		\begin{aligned}
			| \delta^k_\nu \Delta_\nu^{-|k|/2}(x,y)&-\delta^k_\nu \Delta_\nu^{-|k|/2}(x,y')|\\
			&= c\Big|\int_0^\vc t^{|k|/2}\big[\delta^k_\nu p_t^\nu(x,y)-\delta^k_\nu p_t^\nu(x,y')\big] \f{dt}{t}\Big|\\
			&\lesi |y-y'|\int_0^\vc t^{|k|/2}\sup_{\theta\in[0,1]}\big|\partial_y \delta^k_\nu p_t^\nu(x,y +\theta(y-y'))\big| \f{dt}{t}.
		\end{aligned}
	\end{equation}
	By Theorem \ref{prop-gradient x y d>2}, we have
	\[
	\begin{aligned}
		\big|\partial_y \delta^k_\nu p_t^\nu(x,y +\theta(y-y'))\big|\lesi &\Big[\f{1}{\sqrt t} +\f{1}{\rho(y +\theta(y-y'))}\Big]\f{1}{t^{(n+k)/2}}\exp\Big(-\f{|x-[y +\theta(y-y')]|}{ct}\Big)\\
		&\ \ \ \ \times \Big(1+\f{\sqrt t}{\rho(x)}+\f{\sqrt t}{\rho(y +\theta(y-y'))}\Big)^{-(\nu_{\rm min}+1/2)}.
	\end{aligned}
	\]
	Note that 
	\[
	|x-[y +\theta(y-y')]|\sim |x-y|\ \ \text{and} \ \ \rho(y +\theta(y-y'))\sim \rho(y)
	\]
	for all $\theta\in [0,1]$, $|y-y'|\le \f{1}{2}|x-y|$ and $|y-y'|\lesi \max\{\rho(y),\rho(y')\}$.
	
	Therefore,
	\[
	\big|\partial_y \delta^k_\nu p_t^\nu(x,y +\theta(y-y'))\big|\lesi \Big[\f{1}{\sqrt t} +\f{1}{\rho(y)}\Big]\f{1}{t^{(n+k)/2}}\exp\Big(-\f{|x-y|}{ct}\Big)\ \Big(1+\f{\sqrt t}{\rho(x)}+\f{\sqrt t}{\rho(y)}\Big)^{-(\nu_{\rm min}+1/2)}
	\]
	for all $\theta\in [0,1]$, $|y-y'|\le \f{1}{2}|x-y|$ and $|y-y'|\lesi \max\{\rho(y),\rho(y')\}$.
	
	Putting it back into \eqref{eq- R(x,y)-R(x,y')}, 
	\[
	\begin{aligned}
		| \delta^k_\nu \Delta_\nu^{-|k|/2}(x,y)&-\delta^k_\nu \Delta_\nu^{-|k|/2}(x,y')|\\
		&\lesi |y-y'|\int_0^\vc \f{1}{\rho(y)}\Big]\f{1}{t^{(n+k)/2}}\exp\Big(-\f{|x-y|}{ct}\Big)\ \Big(1+\f{\sqrt t}{\rho(x)}+\f{\sqrt t}{\rho(y)}\Big)^{-(\nu_{\rm min}+1/2)}  \f{dt}{t}\\						
		&\lesi  |y-y'|\Big[\sum_{j=1}^n\f{1}{\rho(y)}+ \f{1}{|x-y|}\Big]\Big(1+\f{|x-y|}{\rho(x)}+\f{|x-y|}{\rho(y)}\Big)^{-(\nu_{\min}+1/2)}\f{1}{|x-y|^n}\\
		&\lesi \Big[ \f{|y-y'|}{\rho(y)}+ \f{|y-y'|}{|x-y|}\Big]\Big(1+\f{|x-y|}{\rho(x)}+\f{|x-y|}{\rho(y)}\Big)^{-(\nu_{\min}+1/2)}\f{1}{|x-y|^n}\\
		&\lesi \f{|y-y'|}{\rho(y)} \Big(1+\f{|x-y|}{\rho(x)}+\f{|x-y|}{\rho(y)}\Big)^{-(\nu_{\min}+1/2)}\f{1}{|x-y|^n}+ \f{|y-y'|}{|x-y|}\f{1}{|x-y|^n}\\
		&=: E_1 +E_2.
	\end{aligned}
	\]

	Since $|y-y'|\lesi \rho(y)$, we have
	\[
	\begin{aligned}
		E_1&\lesi \Big(\f{|y-y'|}{\rho(y)}\Big)^{\gamma_\nu} \Big(1+\f{|x-y|}{\rho(x)}+\f{|x-y|}{\rho(y)}\Big)^{-\gamma_\nu}\f{1}{|x-y|^n}\\
		&\lesi \Big(\f{|y-y'|}{\rho(y)}\Big)^{\gamma_\nu} \Big(\f{|x-y|}{\rho(x)}\Big)^{-\gamma_\nu}\f{1}{|x-y|^n}\\
		&\lesi \Big(\f{|y-y'|}{|x-y|}\Big)^{\gamma_\nu}\f{1}{|x-y|^n},
	\end{aligned}
	\]
	where  and $\gamma_\nu = \min\{1, \nu_{\min}+1/2\}$.
	
	For the same reason, since $|y-y'|\le |x-y|/2$, we have
	\[
	\begin{aligned}
		E_2	&\lesi \Big(\f{|y-y'|}{|x-y|}\Big)^{\gamma_\nu}\f{1}{|x-y|^n}.
	\end{aligned}
	\] 
	It follows that 
	\[
	\begin{aligned}
		| \delta^k_\nu \Delta_\nu^{-|k|/2}(x,y)-\delta^k_\nu \Delta_\nu^{-|k|/2}(x,y')|\lesi \Big(\f{|y-y'|}{x-y}\Big)^{\gamma_\nu}\f{1}{|x-y|^n},
	\end{aligned}
	\]
	whenever $|y-y'|\le |x-y|/2$.
	
	Similarly, we also have
	\[
	\begin{aligned}
		| \delta^k_\nu \Delta_\nu^{-|k|/2}(y,x)-\delta^k_\nu \Delta_\nu^{-|k|/2}(y',x)|&\lesi \Big(\f{|y-y'|}{|x-y|}\Big)^{\gamma_\nu}\f{1}{|x-y|^n},
	\end{aligned}
	\]
	whenever $|y-y'|\le |x-y|/2$.
	
	This completes our proof.
\end{proof}

We would like to emphasize that Theorem \ref{thm-Riesz transform} is new, even when $n=1$. In fact, in the case of $n=1$,   the boundedness of the Riesz transform was explored in \cite{NS, Betancor1, Betancor2}. In \cite{NS}, the Riesz operator was decomposed into local and global components. 

We now give the proof of Theorem \ref{thm- boundedness on Hardy and BMO}.

\begin{proof}[Proof of Theorem \ref{thm- boundedness on Hardy and BMO}:]
Fix $\f{n}{n+\gamma_\nu}<p\le 1$, $k \in \mathbb N^n$ and $M>n(1/p-1)$. 

\noindent (i)  Recall from \cite{SY} that for $p\in (0,1]$ and $N\in \mathbb N$, a function $a$ is call a $(p,N)_{\Delta_\nu}$ atom associated to a ball $B$  if  
\begin{enumerate}[{\rm (i)}]
	\item  $a=\Delta_\nu^N b$;
	\item $\supp \Delta_\nu^{k}b\subset B, \ k=0, 1, \dots, M$;
	\item $\|\Delta_\nu^{k}b\|_{L^\vc(\mathbb{R}^n_+)}\leq
	r_B^{2(N-k)}|B|^{-\f{1}{p}},\ k=0,1,\dots,N$.
\end{enumerate}

Let $f\in H^p_{\Delta_\nu}(\mathbb{R}^n_+)\cap L^2(\Rn_+)$. Since $\Delta_\nu$ is a nonnegative self-adjoint operator and satisfies the Gaussian upper bound, by Theorem 1.3 in \cite{SY}, we can write $f = \sum_{j}\lambda_ja_j$ in $L^2(\Rn_+)$, where 
$\sum_{j}|\lambda_j|^p\sim \|f\|^p_{H^p_{\Delta_\nu}(\mathbb{R}^n_+)}$ and each $a_j$ is a $(p,N)_{\Delta_\nu}$ atom with $N>n(\f{1}{p}-1)$.  

In addition, by  Theorem \ref{mainthm2s}, $H^p_{\Delta_{\nu+2\vec{M} }}(\mathbb{R}^n_+)\equiv H^p_{\Delta_\nu}(\mathbb{R}^n_+)\equiv H^p_{\rho}(\mathbb R^n_+)$, where $\vec{M}=(M,\ldots,M)\in \mathbb R^n$. Consequently, it suffices to prove that 
\[
\Big\|\sup_{t>0}|e^{-t\Delta_{\nu + k+ 2\vec{M}}}\delta^k_\nu  \Delta_\nu^{-|k|/2}a|\Big\|_p\lesi 1
\]
for all $(p,M)_{\Delta_\nu}$ atoms $a$.

Let $a$ be a $(p,M)_{\Delta_\nu}$ atom associated to a ball $B$. We have
\[
\begin{aligned}
	\Big\|\sup_{t>0}|e^{-t\Delta_{\nu + k+ 2\vec{M}}}\delta^k_\nu  \Delta_\nu^{-|k|/2}a|\Big\|_p&\lesi \Big\|\sup_{t>0}|e^{-t\Delta_{\nu + k+ 2\vec{M}}}\delta^k_\nu  \Delta_\nu^{-|k|/2}a|\Big\|_{L^p(4B)}\\ & \ \ \ \ \ +\Big\|\sup_{t>0}|e^{-t\Delta_{\nu + k+ 2\vec{M}}}\delta^k_\nu  \Delta_\nu^{-|k|/2}a|\Big\|_{L^p(\mathbb R^n_+\backslash 4B)}.		
\end{aligned}
\]
Using the $L^2$-boundedness of both $f\mapsto \sup_{t>0}|e^{-t \Delta_{\nu + \vec M} }f|$ and the Riesz transform $\delta^k_\nu  \Delta_\nu^{-|k|/2}$ and the H\"older  inequality, by the standard argument, we have
\[
\Big\|\sup_{t>0}|e^{-t\Delta_{\nu + k+ 2\vec{M}}}\delta^k_\nu  \Delta_\nu^{-|k|/2}a|\Big\|_{L^p(4B)}\lesi 1.
\]
For the second term, using $a=\Delta_\nu^Mb$,
\[
e^{-t\Delta_{\nu + k+ 2\vec{M}}}\delta^k_\nu  \Delta_\nu^{-|k|/2}a = e^{-t\Delta_{\nu + k+ 2\vec{M}}}\delta^k_\nu  \Delta_\nu^{M-|k|/2}b.
\]
We have
\[
\begin{aligned}
	K_{e^{-t\Delta_{\nu + k+ 2\vec{M}}}\delta^k_\nu  \Delta_\nu^{M-|k|/2}}(x,y) &= c\int_0^\vc u^{|k|/2}K_{e^{-t\Delta_{\nu + k+ 2\vec{M}}}\delta^k_\nu  \Delta_\nu^{M}e^{-u\Delta_\nu}}(x,y) \f{du}{u}\\
	&= c\int_0^t \ldots \f{du}{u} + c\int_t^\vc \ldots \f{du}{u}\\
	&=I_1+I_2,
\end{aligned}
\]
where $K_{e^{-t\Delta_{\nu + k+ 2\vec{M}}}\delta^k_\nu  \Delta_\nu^{M-|k|/2}}(x,y)$ is the kernel of ${e^{-t\Delta_{\nu + k+ 2\vec{M}}}\delta^k_\nu  \Delta_\nu^{M-|k|/2}}$.

For $I_1$, by Corollary \ref{cor1  d ge 2} and Theorem \ref{thm-heat kernel} we have, for $u\le t$,
\[
\begin{aligned}
	|K_{e^{-t\Delta_{\nu + k+ 2\vec{M}}}\delta^k_\nu  \Delta_\nu^{M}e^{-u\Delta_\nu}}(x,y)|&=\Big|\int_{\Rn_+} \Delta_\nu^M(\delta_\nu^*)^kp_t^{\nu+k+2\vec{M}}(z,x)p_u^{\nu}(z,y)dz\Big|\\
	&\lesi \f{1}{t^{M+|k|/2}}\int_{\Rn_+}\f{1}{t^{n/2}}\exp\Big(
	-\f{|x-z|^2}{ct}\Big)\f{1}{u^{n/2}}\exp\Big(
	-\f{|z-y|^2}{ct}\Big)dz\\
	&\lesi \f{1}{t^{M+|k|/2+n/2}}\exp\Big(
	-\f{|x-y|^2}{ct}\Big),
\end{aligned}
\]
which implies that 
\[
\begin{aligned}
	I_1&\lesi \f{1}{t^{M+n/2}}\exp\Big(
	-\f{|x-y|^2}{ct}\Big) \\
	&\lesi \f{1}{|x-y|^{2M+n}}.
\end{aligned}
\]
Similarly, by Corollary \ref{cor1  d ge 2} and Theorem \ref{thm-heat kernel} we have, for $u> t$
\[
\begin{aligned}
	|K_{e^{-t\Delta_{\nu + k+ 2\vec{M}}}\delta^k_\nu  \Delta_\nu^{M}e^{-u\Delta_\nu}}(x,y)|&= \Big|\int_{\Rn_+}p_t^{\nu+k+2\vec{M}}  (x,z) \delta^k_\nu  \Delta_\nu^{M} p_u^\nu(z,y)| dz\Big|\\
	&	\lesi \f{1}{u^{M+|k|/2}}\int_{\Rn_+}\f{1}{t^{n/2}}\exp\Big(
	-\f{|x-z|^2}{ct}\Big)\f{1}{u^{n/2}}\exp\Big(
	-\f{|z-y|^2}{ct}\Big)dz\\
	&\lesi \f{1}{u^{M+|k|/2+n/2}}\exp\Big(
	-\f{|x-y|^2}{cu}\Big)
\end{aligned}
\]
which implies that 
\[
\begin{aligned}
	I_2
	&\lesi \int_t^\vc \f{1}{u^{M}} \f{1}{u^{n/2}}\exp\Big(-\f{|x-y|^2}{cu}\Big)  \f{du}{u}\\
	&\lesi \int_0^{|x-y|^2}\ldots + \int_{|x-y|^2}^\vc\ldots \\ 		
	&\lesi \f{1}{|x-y|^{2M}} \f{1}{|x-y|^n}.
\end{aligned}
\]

Hence,
\[
\begin{aligned}
	\Big\|\sup_{t>0}|e^{-t\Delta_{\nu + k+ 2\vec{M}}}\delta^k_\nu  \Delta_\nu^{-|k|/2}a|\Big\|^p_{L^p(\mathbb R^n_+\backslash 4B)}
	& \lesi \int_{\mathbb R^n_+\backslash 4B} \Big[\int_{B} \f{1}{|x-y|^{2M}} \f{1}{|x-y|^n}  |b(y)|dy\Big]^{p}dx\\
	& \lesi \int_{\mathbb R^n_+\backslash 4B} \Big[\int_{B} \f{1}{|x-x_B|^{2M}} \f{1}{|x-x_B|^n} |b(y)|dy\Big]^{p}dx\\
	&\lesi \int_{\mathbb R^n_+\backslash 4B}  \f{\|b\|_1^p}{|x-x_B|^{(n+2M)p}}  dx\\
	&\lesi \int_{\mathbb R^n_+\backslash 4B}  \f{r_B^{2Mp}|B|^{p-1}}{|x-x_B|^{(n+2M)p}}  dx\\		
	&\lesi 1,
\end{aligned}
\]
as long as $M>n(1/p-1)$.
\bigskip

\noindent (ii)  By the duality in Theorem \ref{mainthm-dual}, it suffices to prove that the conjugate $\Delta_\nu^{-1/2}\delta_{\nu_j}^*$ is bounded on the Hardy space  $H^p_{\rho}(\mathbb R^n_+)$. By Theorem \ref{mainthm2s}, we need only to prove that 
\[
\Big\|\sup_{t>0}|e^{-t\Delta_{\nu}} \Delta_\nu^{-|k|/2}(\delta_\nu^*)^ka|\Big\|_p\lesi 1
\]
for all $(p,\rho)$ atoms $a$.

Suppose that  $a$ is a $(p,\rho)$ atom associated to a ball $B$. Then we can write
\[
\begin{aligned}
	\Big\|\sup_{t>0}|e^{-t\Delta_{\nu}} \Delta_\nu^{-|k|/2}(\delta_\nu^*)^ka|\Big\|_p&\lesi \Big\|\sup_{t>0}|e^{-t\Delta_{\nu}} \Delta_\nu^{-|k|/2}(\delta_\nu^*)^ka|\Big\|_{L^p(4B)} +\Big\|\sup_{t>0}|e^{-t\Delta_{\nu}}\Delta_\nu^{-|k|/2}(\delta_\nu^*)^ka|\Big\|_{L^p(\mathbb R^n_+\backslash 4B)}.		
\end{aligned}
\]
Since  $f\mapsto \sup_{t>0}|e^{-t \Delta_{\nu}}f|$ and  $\Delta_\nu^{-|k|/2}(\delta_\nu^*)^k$ are bounded on $L^2(\mathbb R^n_+)$, by  the H\"older's inequality and the standard argument, we have
\[
\Big\|\sup_{t>0}|e^{-t\Delta_{\nu}} \Delta_\nu^{-|k|/2}(\delta_\nu^*)^ka|\Big\|_{L^p(4B)}\lesi 1.
\]
For the second term, we consider two cases.

\textbf{Case 1: $r_B=\rho(x_B)$.} Using \eqref{eq-sub formula}, we have for $x\in (4B)^c$,
\[
\begin{aligned}
	\sup_{t>0}|e^{-t\Delta_{\nu}} \Delta_\nu^{-|k|/2}(\delta_\nu^*)^ka(x)|&=\sup_{t>0}|(\delta_\nu^k\Delta_\nu^{-|k|/2} e^{-t\Delta_{\nu}} )^*a(x)|\\
	&=c\sup_{t>0}\Big|\int_0^\vc  u^{|k|/2}(\delta_\nu^k\Delta_\nu^{-|k|/2} e^{-(t+u)\Delta_{\nu}} )^* a(x)\f{du}{u}\Big|\\
	&=c\sup_{t>0}\Big|\int_0^\vc \int_{B} u^{|k|/2}\delta_\nu^k  p_{t+u}^\nu(y,x) a(y)dy\f{du}{u}\Big|\\
	&\lesi \sup_{t>0}\Big|\int_0^\vc \int_{B} |u^{|k|/2}\delta_\nu^k p_{t+u}^\nu(y,x)| |a(y)|dy\f{du}{u}\Big|.
\end{aligned}
\]
Using Theorem \ref{thm-delta k pt nu} and the fact $\rho(y)\sim \rho(x_B)$ and $|x-y|\sim |x-x_B|$ for $y\in B$ and $x\in \Rn_+\setminus 4B$, we further obtain
\[
\begin{aligned}
		\sup_{t>0}&|e^{-t\Delta_{\nu}} \Delta_\nu^{-|k|/2}(\delta_\nu^*)^ka(x)|\\
		& 
	\lesi \sup_{t>0} \int_0^\vc \int_{B} \f{u^{|k|/2}}{(u+t)^{(n+|k|)/2}}  \exp\Big(-\f{|x-x_B|^2}{c(u+t)}\Big)\Big(\f{\sqrt{u+t}}{\rho(x_B)}\Big)^{-\gamma_\nu} |a(y)|dy\f{du}{u}\\
	& 
	\lesi \|a\|_1\sup_{t>0} \int_0^\vc  \f{1}{(u+t)^{n/2}}  \exp\Big(-\f{|x-x_B|^2}{c(u+t)}\Big)\Big(\f{\sqrt{u+t}}{\rho(x_B)}\Big)^{-\gamma_\nu}  \f{du}{u}\\
	&\lesi   \f{1}{|x-x_B|^n}\Big(\f{\rho(x_B)}{|x-x_B|}\Big)^{\gamma_\nu}\|a\|_1\\
	&\lesi   \f{r_B^{\gamma_\nu}}{|x-x_B|^{n+\gamma_\nu}} |B|^{1-1/p}.
\end{aligned}
\]

It follows that
\[
\Big\|\sup_{t>0}|e^{-t\Delta_{\nu}} \Delta_\nu^{-|k|/2}(\delta_\nu^*)^ka|\Big\|_{L^p(\mathbb R^n_+\backslash 4B)}\lesi 1,
\]
as long as $\f{n}{n+\gamma_\nu}<p\le 1$.

\bigskip

\textbf{Case 2: $r_B<\rho(x_B)$.} 

Using the formula  \eqref{eq-sub formula}, we have
\[
\begin{aligned}
	e^{-t\Delta_{\nu}} \Delta_\nu^{-|k|/2}(\delta_\nu^*)^ka(x)&=c\int_0^\vc  u^{|k|/2}e^{-(t+u)\Delta_\nu}(\delta_\nu^*)^k a(x)\f{du}{u}\\
	&=c\int_0^\vc  u^{|k|/2}[\delta_\nu^ke^{-(t+u)\Delta_\nu}]^* a(x)\f{du}{u}\\
	&=c\int_0^\vc\int_{\Rn_+}  u^{|k|/2} \delta_\nu^kp_{t+u}^\nu (y,x) a(y)dy\f{du}{u}.
\end{aligned}
\]
Using the cancellation property $\displaystyle \int a(x)x^\alpha dx= 0$ for all $\alpha$ with $|\alpha|\le N:=\lfloor n(1/p-1)\rfloor$, we have
\[
\begin{aligned}
	\big|e^{-t\Delta_{\nu}} &\Delta_\nu^{-|k|/2}(\delta_\nu^*)^ka(x)\big|\\
	&=c\Big|\int_0^\vc\int_{\Rn_+}  u^{|k|/2} \Big[\delta_\nu^kp_{t+u}^\nu (y,x)-\sum_{|\alpha|\le N}\f{(y-x_B)^\alpha}{\alpha!}\partial^\alpha \delta_\nu^kp_{t+u}^\nu (x_B,x) \Big] a(y)dy\f{du}{u}\Big|\\
	&\lesi \int_0^\vc\int_{\Rn_+}  u^{|k|/2}   |y-x_B|^{N+1} \sup_{\theta\in [0,1]\atop |\beta|= N+1}|\partial^\beta \delta_\nu^kp_{t+u}^\nu \big(y+\theta(x_B-y),x\big)| |a(y)|dy\f{du}{u}\\
	&\lesi \int_0^\vc\int_{\Rn_+}  u^{|k|/2}   r_B^{N+1} \sup_{\theta\in [0,1]\atop |\beta|= N+1}|\partial^\beta \delta_\nu^kp_{t+u}^\nu \big(y+\theta(x_B-y),x\big)| |a(y)|dy\f{du}{u}
\end{aligned}
\]
Note that for $y\in B$, $x\in \Rn_+\backslash 4B$ and $\theta \in [0,1]$, we have $y+\theta(x_B-y)\in B$ and hence $\rho(y+\theta(x_B-y))\sim \rho(x_B)$ and $|x-[y+\theta(x_B-y)]|\sim |x-y|\sim |x-x_B|$. This, together with Proposition \ref{prop-gradient x y d>2}, implies, for $x\in \Rn_+\backslash 4B$ and $y\in B$,
\[
\begin{aligned}
	\sup_{\theta\in [0,1]\atop |\beta|= N+1}&|\partial^\beta \delta_\nu^kp_{t+u}^\nu \big(y+\theta(x_B-y),x\big)|\\
	&\lesi \Big(\f{1}{(t+u)^{(N+1)/2}}+\f{1}{\rho(x_B)^{N+1}}\Big)\f{1}{(t+u)^{(n+k)/2}}\exp\Big(-\f{|x-x_B|^2}{c(t+u)}\Big)\Big(\f{\sqrt{t+u}}{\rho(x_B)}\Big)^{-\gamma_\nu}\\
	&\lesi  \Big(\f{1}{|x-x_B|^{N+1}}+\f{1}{\rho(x_B)^{N+1}}\Big)\f{1}{(t+u)^{(n+k)/2}} \exp\Big(-\f{|x-x_B|^2}{2c(t+u)}\Big)\Big(1+\f{\sqrt{t+u}}{\rho(x_B)}\Big)^{-\gamma_\nu}.
\end{aligned}
\] 
Hence,  for $x\in \Rn_+\backslash 4B$ and $y\in B$ we have
\[
\begin{aligned}
		&\big|e^{-t\Delta_{\nu}} \Delta_\nu^{-|k|/2}(\delta_\nu^*)^ka(x)\big|\\
		&\lesi \|a\|_1\int_0^\vc   \f{u^{|k|/2}}{(t+u)^{(n+k)/2}}     \Big(\f{r_B^{N+1}}{|x-x_B|^{N+1}}+\f{r_B^{N+1}}{\rho(x_B)^{N+1}}\Big) \exp\Big(-\f{|x-x_B|^2}{2c(t+u)}\Big) \Big(1+\f{\sqrt{t+u}}{\rho(x_B)}\Big)^{-\gamma_\nu}\f{du}{u}\\
		&\lesi \|a\|_1\int_0^\vc   \f{1}{(t+u)^{n/2}}     \Big(\f{r_B^{N+1}}{|x-x_B|^{N+1}}+\f{r_B^{N+1}}{\rho(x_B)^{N+1}}\Big) \exp\Big(-\f{|x-x_B|^2}{2c(t+u)}\Big) \Big(1+\f{\sqrt{t+u}}{\rho(x_B)}\Big)^{-\gamma_\nu}\f{du}{u}\\
		&\lesi \|a\|_1 \Big(\f{r_B^{N+1}}{|x-y|^{N+1}}+\f{r_B^{N+1}}{\rho(x_B)^{N+1}}\Big) \Big(1+\f{|x-x_B|}{\rho(x_B)}\Big)^{-\gamma_\nu},
\end{aligned}
\]
which implies
\[
\begin{aligned}
	\sup_{t>0}\big|e^{-t\Delta_{\nu}} \Delta_\nu^{-|k|/2}(\delta_\nu^*)^ka(x)\big|&\lesi \|a\|_1 \Big(\f{r_B^{N+1}}{|x-x_B|^{N+1}}+\f{r_B^{N+1}}{\rho(x_B)^{N+1}}\Big) \Big(1+\f{|x-x_B|}{\rho(x_B)}\Big)^{-\gamma_\nu}\\
	&\lesi \|a\|_1 \Big[\f{r_B^{N+1}}{|x-x_B|^{N+1}}+\f{r_B^{N+1}}{\rho(x_B)^{N+1}} \Big(\f{\rho(x_B)}{|x-x_B|}\Big)^{\gamma_\nu}\Big].
\end{aligned}
\]
Since $r_B\le \min\{\rho(x_B), |x-x_B|\}$, this further implies
\[
\begin{aligned}
	\sup_{t>0}\big|e^{-t\Delta_{\nu}} \Delta_\nu^{-|k|/2}(\delta_\nu^*)^ka(x)\big|&\lesi \|a\|_1 \Big[\f{r_B^{N+1}}{|x-y|^{N+1}}+\Big(\f{r_B}{\rho(x_B)}\Big)^{\gamma_\nu\wedge (N+1)} \Big(\f{\rho(x_B)}{|x-x_B|}\Big)^{\gamma_\nu\wedge (N+1)}\Big]\\
	&\lesi \|a\|_1 \Big[\f{r_B^{N+1}}{|x-y|^{N+1}}+\Big(\f{r_B}{|x-x_B|}\Big)^{\gamma_\nu\wedge (N+1)} \Big]
\end{aligned}
\]
for $x\in \Rn_+\backslash 4B$ and $y\in B$.

Consequently,
\[
\Big\|\sup_{t>0}|e^{-t\Delta_{\nu}} \Delta_\nu^{-|k|/2}(\delta_\nu^*)^ka|\Big\|_{L^p(\mathbb R^n_+\backslash 4B)}\lesi 1,
\]
as long as $\f{n}{n+\gamma_\nu}<p\le 1$.

This completes our proof.

\end{proof}
{\bf Acknowledgement.} The author was supported by the research grant ARC DP140100649 from the Australian Research Council.

\end{document}